\documentclass[a4paper]{article}
\usepackage[utf8]{inputenc}
\usepackage{amsthm}
\usepackage{amsmath}
\usepackage{amssymb}

\usepackage[top=3cm, bottom=3.5cm, left=3cm, right=3cm]{geometry}
\usepackage{comment}

\usepackage{graphicx}
\graphicspath{ {./images/} }

\usepackage{mathtools}

\usepackage{tikz}
\usepackage{tikz-3dplot}

\usepackage[graphicx]{realboxes}
\usepackage{caption}

\usepackage[plainpages=false,pdfpagelabels,colorlinks=true,citecolor=blue,hypertexnames=false]{hyperref}
\usepackage[capitalise]{cleveref}

\newtheorem{thm}{Theorem}

\newtheorem{lem}[thm]{Lemma}
\newtheorem{dfn}[thm]{Definition}
\newtheorem{corr}[thm]{Corollary}
\newtheorem{example}{Example}

\DeclareMathOperator{\Q}{\mathbb{Q}}
\DeclareMathOperator{\R}{\mathbb{R}}

\DeclareMathOperator{\PPP}{\mathbf{P}}
\DeclareMathOperator{\OOO}{\mathbf{O}}
\DeclareMathOperator{\PP}{\mathcal{P}}
\DeclareMathOperator{\QQ}{\mathcal{Q}}
\DeclareMathOperator{\id}{\mathrm{Id}}
\DeclareMathOperator{\spanp}{\mathrm{span}^+}

\DeclareMathOperator{\NOP}{\mathbf{NOP}}
\DeclareMathOperator{\RUP}{\mathbf{RUP}}
\DeclareMathOperator{\RID}{\mathbf{RID}}

\DeclareMathOperator{\Circ}{\mathrm{Disc}}
\DeclareMathOperator{\Sect}{\mathrm{Sect}}

\newcommand{\dd}{\mathrm{d}}

\newcommand{\thetab}{\overline{\theta}}
\newcommand{\phib}{\overline{\varphi}}
\newcommand{\alphab}{\overline{\alpha}}
\newcommand{\Mib}{\overline{M_1}}
\newcommand{\Miib}{\overline{M_2}}
\newcommand{\Xib}{\overline{X_1}}
\newcommand{\Xiib}{\overline{X_2}}
\renewcommand{\phi}{\varphi}
\renewcommand{\epsilon}{\varepsilon}

\newcommand{\ssin}{\sin_{\mathbb{Q}}}
\newcommand{\scos}{\cos_{\mathbb{Q}}}

\theoremstyle{remark}

\title{A convex polyhedron without Rupert's property}
\author{Jakob Steininger\thanks{\href{https://www.statistik.at/en}{Statistics Austria}, \textit{Vienna}, \textit{Austria}. \href{mailto:steininger.jakob@yahoo.com}{steininger.jakob@yahoo.com}} \, and \, Sergey Yurkevich\thanks{\href{https://www.artech.at/}{A\&R TECH}, \textit{Vienna}, \textit{Austria}. \href{mailto:sergey.yurkevich@univie.ac.at}{sergey.yurkevich@univie.ac.at}}}

\date{\today}

\begin{document}
\maketitle

\begin{abstract}
    A three-dimensional convex body is said to have Rupert's property if its copy can be passed through a straight hole inside that body. In this work we construct a polyhedron which is provably not Rupert, thus we disprove a conjecture from 2017. We also find a polyhedron that is Rupert but not locally Rupert. 
\end{abstract}

\section{Introduction} \label{sec:intro}

In \cite[pp. 470]{Wallis1685} John Wallis proved a prediction by Prince Rupert of the Rhine concerning an interesting property of the three-dimensional cube: It is possible to cut a hole inside this solid such that an identical cube can pass through this hole. Formulated like that, this property sounds quite surprising, however the following equivalent formulation feels more believable: There exist two (orthogonal) projections of the 3D cube onto 2D, such that one of the resulting polygons (convex hulls) fits inside the other. For a visual illustration of this property for the cube, see Figure~\ref{fig:rupert_cube2}.

At least since the 20th century several mathematicians wondered which other solids have ``Rupert's property'' (\cref{def:rupert} contains the precise formulation). For instance, in \cite{Scriba68, JeWeYu17} it is proven that all Platonic solids are \emph{Rupert}, and the authors of \cite{JeWeYu17} asked whether all convex polyhedra in $\R^3$ have this property. Motivated by this conjecture, several articles looked at Archimedean, Catalan and Johnson solids from the perspective of Rupert's property  \cite{ChYaZa18,Tonpho18,Hoffmann19,Lavau19,StYu23, Fredriksson22}. As of today, most of these polyhedra are known to have Rupert's property, but some remain open. Several interesting approaches and connections were discovered, for instance a mysterious relationship between duality and Rupert's property. We refer to \cite{StYu23} for more details and a more thorough historical overview.

In this work we \emph{disprove} the aforementioned conjecture that all convex three-dimensional polyhedra are Rupert. More precisely, we explicitly construct a polyhedron which we call \emph{Noperthedron} ($\NOP$ for short, see Figure~\ref{fig:noperthedron}) and prove: 
\begin{thm} \label{thm:main}
    The Noperthedron does not have Rupert's property. 
\end{thm}

\begin{figure}[htb]
    \begin{minipage}{0.48\textwidth}
        \centering
        \begin{tikzpicture}[scale=1]
            \draw[thick] (0,1.6329932) -- (1.4142,0.8164932) -- (1.4142,-0.8165) -- (0,-1.6329932) -- (-1.4142,-0.8164932) -- (-1.4142,0.8165) -- (0,1.6329932) -- (1.4142,0.8164932);
            \draw[thick] (-1.4142,0.8165) -- (0,0) -- (1.4142,0.8164932);
            \draw[thick] (0,0) -- (0,-1.6329932);
            \draw[dashed, gray] (0,0) -- (0,1.6329932);
            \draw[dashed, gray] (-1.4142,-0.8165) -- (0,0) -- (1.4142,-0.8164932);
            \draw[thick, blue] (-1,1) -- (1,1) -- (1,-1) -- (-1,-1) -- (-1,1) -- (1,1);
            \fill[blue!60, opacity=0.2] (-1,1) -- (1,1) -- (1,-1) -- (-1,-1) -- (-1,1);
        \end{tikzpicture}
        \caption{Two projections of the unit cube.}
        \label{fig:rupert_cube2}
    \end{minipage}\hfill
    \begin{minipage}{0.48\textwidth}
        \centering
        \includegraphics[width=0.5\linewidth]{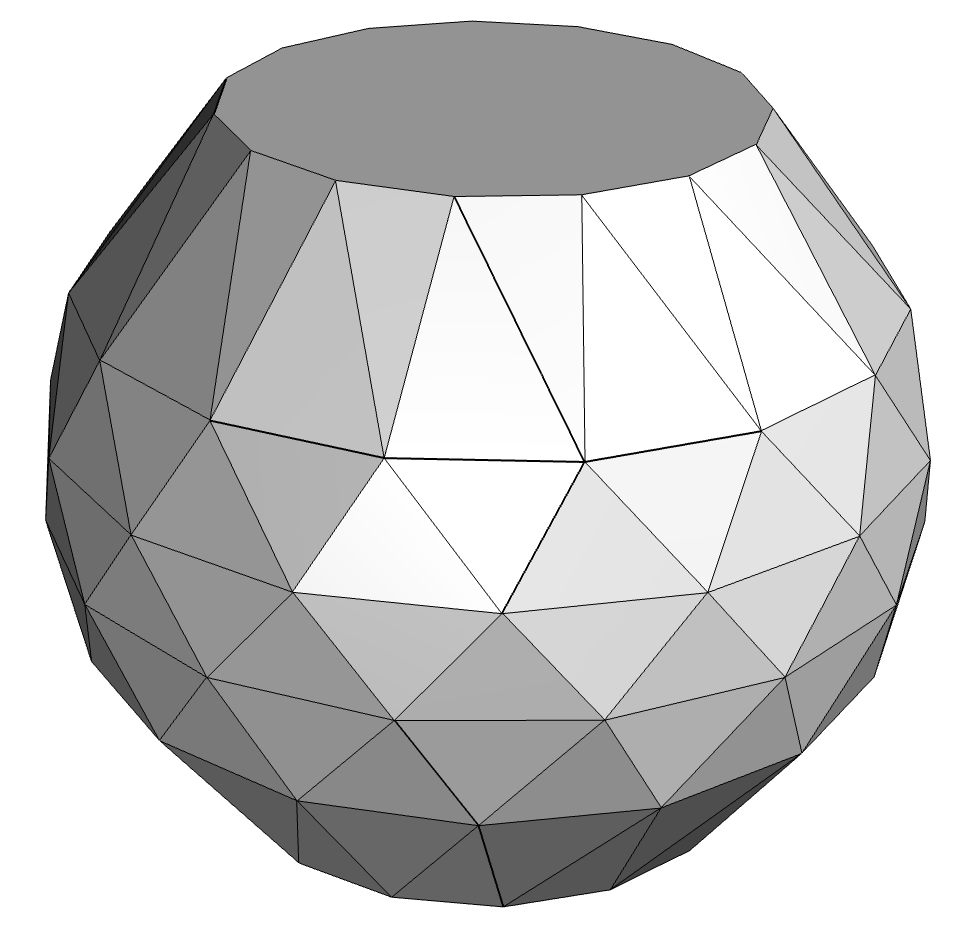} 
        \caption{The Noperthedron.}
        \label{fig:noperthedron}
    \end{minipage}
\end{figure}

\subsection{Setting and definitions}

In this paper a \emph{polyhedron}, usually denoted by $\PPP$, is a finite non-degenerate set of points in $\R^3$ in convex position. We will denote by $\overline{\PPP} \subseteq \R^3$ the smallest convex set containing all points of a polyhedron $\PPP$ (i.e. including its interior)
and by $\PPP^\circ$ the interior of $\overline{\PPP}$. Further, a polyhedron $\PPP$ is called point-symmetric (with respect to the origin) if $\PPP = -\PPP$, i.e. $-P \in \PPP$ for all $P \in \PPP$. Note that all polyhedra in this article will have this property. We will say that such a polyhedron $\PPP$ has radius $\rho > 0$ if for all $P \in \PPP$ it holds that $\|P\| \leq \rho$ and for some $P \in \PPP$ equality~holds.

A polygon, usually denoted by $\PP$ or $\QQ$, is a finite set of points in $\R^2$ that not all lie on the same line. Similar to polyhedra, we denote by $\overline{\PP}$ the convex hull of $\PP$ and by $\PP^\circ$ the interior of~ $\overline{\PP}$.

In order to define Rupert's property we first introduce the rotation matrix $R(\alpha)$, a parametrization of the unit sphere $X(\theta, \phi)$ and $M(\theta, \phi)$ -- the orthogonal projection from $\R^3$ to $\R^2$ in direction of $X(\theta, \phi)$:

\begin{dfn} \label{def:RXM}
For any $\alpha, \theta, \varphi \in \R$ we define: 
    \begin{align*}
    R(\alpha)&\coloneqq\begin{pmatrix} 
    \cos\alpha & -\sin\alpha \\ 
    \sin\alpha& \cos\alpha\\ 
    \end{pmatrix} \in \R^{2 \times 2},\\
    X({\theta,\varphi})&\coloneqq(\cos \theta \sin \varphi,\sin \theta \sin \varphi, \cos \varphi)^t \in \R^{3} \cong \R^{3 \times 1},\\
    M(\theta, \phi) &\coloneqq
    \begin{pmatrix} 
    -\sin\theta & \cos\theta & 0 \\ 
    -\cos\theta\cos\varphi & -\sin\theta\cos\varphi & \sin\varphi
    \end{pmatrix} \in \R^{2 \times 3}.
    \end{align*}
\end{dfn}
We let $R(\alpha)$ and $M(\theta,\phi)$ act on \emph{sets of points} so that $R(\alpha) \PP$ is a rotation of a polygon $\PP$ and $M(\theta,\phi) \PPP$ is a projection of a polyhedron $\PPP$. In \cite[Prop. 2]{StYu23} the following property is proven to be equivalent to the original statement of Rupert's property for pointsymmetric polyhedra. We will use it as the \emph{definition} of Rupert's property in this paper:

\begin{dfn} \label{def:rupert}
    A pointsymmetric convex polyhedron $\PPP$ \emph{has Rupert's property} (or \emph{``is Rupert''}) if there exist $\theta_1,\theta_2\in[0,2\pi)$, $\varphi_1,\varphi_2\in[0,\pi]$ and $\alpha \in [-\pi, \pi)$ such that
\begin{align} \label{eq:rupert_condition}
R(\alpha) M(\theta_1, \phi_1)\PPP \subset (M(\theta_2, \phi_2)\PPP)^\circ.
\end{align}
\end{dfn}

Note that in this definition $\PPP$ is assumed to pointsymmetric, therefore, as proven in~\cite[Prop.~2]{StYu23}, no translation is needed. 
As $\PPP$ will always be pointsymmetric for us, we stick to this definition. 
The pointsymmetric property also immediately implies that the region for $\alpha$ can be equivalently reduced to $[-\pi/2, \pi/2)$, as we will show in \cref{lem:symmetries}.

\subsection{Structure of the paper}
The paper is organized as follows: In \cref{sec:outline} we briefly explain the strategy of our proof; in particular we explain the idea behind the \emph{global} and \emph{local} theorems. \cref{sec:noperthedron} contains the definition of the Noperthedron as well as proofs of some of its properties. In \cref{sec:bounds} we prove sharp bounds for operations containing $R(\alpha)$ and $M(\theta,\phi)$, for instance, the essential \cref{lem:sqrt2} and \cref{lem:sqrt5} which state that
\begin{align} \label{eq:sqrt2sqrt5intro}
    \| M(\theta, \phi) - M(\thetab, \phib) \| \leq \sqrt{2}\epsilon \quad \text{and} \quad    \| R(\alpha) M(\theta, \phi) - R(\alphab) M(\thetab, \phib) \| \leq \sqrt{5} \epsilon
\end{align}
if $|\alpha-\alphab|, |\theta-\thetab|,|\varphi-\phib| \leq \varepsilon$.

The Sections \ref{sec:global_thm} and \ref{sec:local_thm} state and prove the global and local theorems respectively. 
Furthermore, as it will be explained in \cref{sec:rational} it is not enough for our purposes to prove the two aforementioned theorems since we need \emph{rational} versions of them for a rigorous computer implementation -- these theorems are stated and proven in \cref{sec:rational}. 
In \cref{sec:algoproof} we explain how we implemented the rational versions of the global and local theorems in \href{https://www.r-project.org/}{R} and \href{https://www.sagemath.org/}{SageMath} -- this, together with the provided source code, completes the proof of our main \cref{thm:main}.
In the penultimate \cref{sec:rupertnotlocal} we construct another polyhedron which is Rupert but not ``locally Rupert''. 
Finally, \cref{sec:discussion} explains how the Noperthedron was found and collects some of the remaining open questions.

\subsection{Outline of the proof} \label{sec:outline}
The overall strategy of this paper for proving that some polyhedron $\PPP$ does not have Rupert's property is to partition the five-dimensional solution space 
\begin{align*} 
    I = [0,2\pi) \times [0, \pi] \times [0,2\pi) \times [0, \pi] \times [-\pi, \pi)
\end{align*}
of \cref{def:rupert} for $(\theta_1,\phi_1,\theta_2,\phi_2,\alpha)$ into small parts and prove for each that no solution in that region exists. Roughly speaking, the latter will be done using the following approach: We show that the ``middle point'' of any region does not yield a solution and argue with effective continuity of the parameters that this also excludes an explicit neighborhood around that point; since we always choose the regions smaller than these neighborhoods, this finally excludes the whole solution space.

Before we explain how to exclude solutions to Rupert's property for regions around a point, we note that the initial interval $I \subseteq \R^5$ can be chosen smaller if $\PPP$ has symmetries. For instance, in \cref{lem:symmetries} we prove that
\[
    M({\theta+2\pi/15,\varphi})\cdot \NOP =M(\theta, \phi) \cdot \NOP
\]
which allows us to conclude that it is enough to consider $[0, 2\pi/15)$ for $\theta_1$ and $\theta_2$. This alone is a reduction of the initial search space by a factor of 225.

For proving that a given $\Psi = (\thetab_1,\phib_1,\thetab_2,\phib_2,\overline{\alpha}) \in \R^5$ does not yield a solution to Rupert's problem for $\PPP$ and excluding a whole region around $\Psi$ we have the \emph{global} (\cref{thm:global}) and the \emph{local} (\cref{thm:local}) theorems at our disposal. Before explaining them, however, let us formulate:
\begin{dfn}
For $\thetab_1,\phib_1,\thetab_2,\phib_2,\overline{\alpha} \in \R$ and $\epsilon>0$ we denote by $[\thetab_1 \pm \epsilon, \phib_1 \pm \epsilon, \thetab_2\pm \epsilon, \phib_2\pm \epsilon, \overline{\alpha} \pm \epsilon]$ the \emph{filled} (closed) 5-dimensional cube with midpoint $(\thetab_1,\phib_1,\thetab_2,\phib_2,\overline{\alpha}) \in \R^5$ and side length $2 \epsilon$.    
\end{dfn}

The \emph{global theorem} (\cref{thm:global}) is intended for two generic projections of $\PPP$: It is based on the fact that if a vertex of the ``smaller'' projection $\PP = R(\alphab) M(\thetab_1, \phib_1)\PPP$ is \emph{strictly outside} the ``larger'' projection $\QQ = M(\thetab_2, \phib_2)\PPP$ then, obviously, $\thetab_1,\phib_1,\thetab_2,\phib_2,\alphab$ is not a solution to Rupert's problem for $\PPP$. Moreover, by continuity it is easy to see that there exists an $\epsilon>0$ such that $[\thetab_1 \pm \epsilon, \phib_1 \pm \epsilon, \thetab_2\pm \epsilon, \phib_2\pm \epsilon, \overline{\alpha} \pm \epsilon]$ has no solutions either. In fact, the distance of that vertex to $\QQ$ and some explicit bounds like (\ref{eq:sqrt2sqrt5intro}) can be used to obtain tight estimates on $\epsilon$. In order to be able to choose $\epsilon$ as big as possible, we will actually use the evaluation of the first and second derivatives of $R(\alpha)$, $M(\theta,\phi)$ at $\Psi$ instead of global estimates like in (\ref{eq:sqrt2sqrt5intro}).

If the projections $\PP$ and $\QQ$ look (almost) exactly the same, for instance if $\thetab_1 \approx \thetab_2$, $\phib_1 \approx \phib_2$ and $\alphab \approx 0$ then the global theorem cannot be used to exclude a region of solutions, since the farthest vertex of $\PP$ from $\QQ$ can be arbitrarily close to $\QQ$. 
In this case we use the \emph{local theorem} (\cref{thm:local}): Roughly speaking, it states that if there exist three points $P_1,P_2,P_3 \in \PPP$ whose projection $M(\thetab, \phib) \cdot \{P_1, P_2, P_3\}$ is part of the convex hull of $\PP \approx \QQ$ and such that the triangle formed by their projection has specific properties, then a region around $(\thetab, \phib, \thetab, \phib, 0)$ cannot have solutions to Rupert's property. 
More precisely, $P_1, P_2, P_3$ must be $\epsilon$-spanning for $(\thetab, \phib)$ meaning that the origin is inside the aforementioned triangle, even after perturbing $\thetab, \phib$ by $\epsilon$ (see \cref{def:eps-spanning}) and the projected points must be $\delta$-\emph{locally maximally distant} for a specific choice of $\delta$ (see \cref{def:LMD}).
The core idea of the local theorem is that, under some natural conditions, the three distances $\|M(\theta, \phi)P_1\|$, $\|M(\theta, \phi)P_2\|$, $\|M(\theta, \phi)P_3\|$ cannot all simultaneously locally increase but, at the same time, if we had a local solution $(\theta_1, \phi_1, \theta_2, \phi_2, \alpha)$ around $(\thetab, \phib, \thetab, \phib, 0)$ then $R(\alpha) M(\theta_1, \phi_1)\PPP \subset (M(\theta_2, \phi_2)\PPP)^\circ$ usually implies that $\|R(\alpha) M(\theta_1, \phi_1) P_i\| = \|M(\theta_1, \phi_1) P_i\| < \|M(\theta_2, \phi_2) P_i\|$ for $i=1,2,3$. 
We use the local theorem more generally, also excluding cases in which $\PP \approx \QQ$ due to symmetries of the Noperthedron; this is why the version presented in \cref{thm:local} might look convoluted at first~sight.

It turns out that it is possible to partition the reduced search space $I$ into roughly 18 million regions such that in each either the global or local theorem applies. We implemented a fast \emph{divide and conquer} algorithm in \href{https://www.r-project.org/}{R} to construct a tree-like structure which stores all regions and the reasons why Rupert's property can be excluded in them (see \cref{sec:algoproof} for more details). However, a final hurdle still remains: using these theorems with naive floating point arithmetic does not guarantee their correctness. Therefore, we prove \cref{thm:global_rational} and \cref{thm:local_rational}, the \emph{rational versions} the global and local theorems. Roughly speaking, they say how to replace all occurring real numbers with rational numbers, i.e. how much rounding affects the inequalities and by how much they need to be refined. We implemented these theorems in the computer algebra software \href{https://www.sagemath.org/}{SageMath}, in order to rigorously confirm that they indeed apply for all 18 million regions and that indeed the union of the regions covers the initial search space. The source code and all the implementations can be found at \href{https://github.com/Jakob256/Rupert}{www.github.com/Jakob256/Rupert}. \\

\noindent
\textbf{Acknowledgments:} We thank Alin Bostan for his everlasting moral support. We also thank \href{https://dwrensha.ws/}{David Renshaw} and \href{https://github.com/jcreedcmu}{Jason Reed} for their \href{https://jcreedcmu.github.io/Noperthedron/}{immense work} on the formalization of this paper in \href{https://lean-lang.org/}{Lean} and the correction of several inaccuracies of the first version of this manuscript. Moreover, the second author is also grateful to Matthias Auchmann and \href{https://www.artech.at/}{A\&R TECH} for financial support as well as some computing resources.

\section{The Noperthedron} \label{sec:noperthedron}
In this section we define the \emph{Noperthedron} $\NOP$ and prove elementary statements about this solid in view of Rupert's property. More precisely, we utilize the symmetries of $\NOP$ in order to show that if that solid was Rupert then the parameters from \cref{def:rupert} could be taken from much smaller~intervals than in the initial definition (\cref{lem:symmetries}).

\subsection{Definition of the Noperthedron}

We start by defining three-dimensional rotation matrices about the $x, y$ and $z$ axes by an angle $\alpha$:
\begin{dfn} \label{def:rotation}
    We define the functions $R_x,R_y,R_z:$  $\mathbb{R}\to \mathbb{R}^{3\times 3}$ by
    \[
            R_x(\alpha)\coloneqq 
        \begin{pmatrix} 
            1 & 0 & 0\\ 
            0 & \cos\alpha & -\sin\alpha\\
            0 & \sin\alpha & \cos\alpha
        \end{pmatrix}, 
        \hspace{1cm}
        R_y(\alpha)\coloneqq 
        \begin{pmatrix} 
            \cos\alpha & 0 & -\sin\alpha\\ 
            0 & 1 & 0\\
            \sin\alpha & 0 & \cos\alpha
        \end{pmatrix},
    \]
    \[
        R_z(\alpha)\coloneqq 
        \begin{pmatrix} 
            \cos\alpha & -\sin\alpha &0\\
            \sin\alpha & \cos\alpha &0\\
            0 & 0 & 1
        \end{pmatrix}.
    \]
\end{dfn}
Note that since these matrices are simply rotations one has $R_a(\alpha)R_a(\beta)=R_a(\alpha+\beta)$ for $a \in \{x,y,z\}$.
We define a representation of the cyclic group $\mathcal{C}_{30}$ by
\[
    \mathcal{C}_{30} \coloneqq \left\{(-1)^\ell R_z\left(\frac{2\pi k}{15}\right) \colon k=0,\dots,14; \ell=0,1\right\}.
\]
Note that $|\mathcal{C}_{30}| = 30$ and we will write $\mathcal{C}_{30} \cdot P = \{c P \,\text{ for } \, c \in \mathcal{C}_{30}\}$ for the orbit of $P$ under the action of $\mathcal{C}_{30}$.

Furthermore, we set three points $C_1,C_2,C_3\in \mathbb{Q}^3$ to be:
\[
    C_1\coloneqq
        \frac{1}{259375205}
        \begin{pmatrix} 
        {152024884} \\ 0 \\ {210152163} 
        \end{pmatrix},
    \qquad
    C_2\coloneqq \frac{1}{10^{10}}
        \begin{pmatrix} 
        6632738028 \\ 6106948881 \\ 3980949609
        \end{pmatrix},
\]
\[
    C_3\coloneqq
        \frac{1}{10^{10}}
        \begin{pmatrix} 
        8193990033 \\ 5298215096 \\ 1230614493
        \end{pmatrix}.
\]
Note that the distance from $C_1$ to the origin is exactly 1, and that $\frac{98}{100} < \|C_i\| < \frac{99}{100}$ for $i=2,3$. 

Now we are ready to define our polyhedron:

\begin{dfn}
    We define the set of points $\NOP \subseteq \mathbb{R}^3$ by the action of $\mathcal{C}_{30}$ on $C_1,C_2,C_3$: 
    \[
        \NOP \coloneqq \mathcal{C}_{30} \cdot C_1 \cup \mathcal{C}_{30}\cdot C_2\cup \mathcal{C}_{30} \cdot C_3,
    \]
    and we call the resulting polyhedron the \emph{Noperthedron}.
\end{dfn}

Clearly, the \emph{Noperthedron} has $3 \cdot 30=90$ vertices. Moreover, since $-\id = (-1)\cdot R_z(0) \in \mathcal{C}_{30}$, $\NOP$ is pointsymmetric with respect to the origin. Finally, all vertices of $\NOP$ have Euclidean norm at most $1$ and at least ${98}/{100}$.

\subsection{Refined Rupert's property for the Noperthedron}
Recall from \cref{def:RXM} that $X({\theta, \varphi}) = (\cos \theta \sin \varphi,\sin \theta \sin \varphi, \cos \varphi)^t$ and that
\begin{align*}
    R(\alpha) =
    \begin{pmatrix} 
        \cos\alpha & -\sin\alpha \\ 
        \sin\alpha& \cos\alpha\\ 
    \end{pmatrix}, 
    \quad 
    M(\theta, \phi) =
    \begin{pmatrix} 
    -\sin\theta & \cos\theta & 0 \\ 
    -\cos\theta\cos\varphi & -\sin\theta\cos\varphi & \sin\varphi
    \end{pmatrix}.
\end{align*}
    
The following three important identities directly follow from the definitions and $\sin x=-\sin(-x)$ as well as $\cos x=\cos(-x)$:
\begin{align} \label{eq:XMRfromR1}
X({\theta, \varphi})^t & =
\begin{pmatrix} 
    0 & 0 & 1\\
\end{pmatrix}
\cdot R_y(\varphi) \cdot R_z(-\theta),\\
M(\theta, \phi) & =
    \begin{pmatrix} 
        0 & 1 & 0\\ 
        -1 & 0 & 0\\
    \end{pmatrix}
    \cdot R_y(\varphi) \cdot R_z(-\theta) \label{eq:XMRfromR2},\\
R(\alpha) M(\theta, \phi) & =
   \begin{pmatrix} 
        0 & 1 & 0\\ 
        -1 & 0 & 0\\
    \end{pmatrix} \cdot
    R_z(\alpha) \cdot R_y(\varphi) \cdot R_z(-\theta). \label{eq:XMRfromR3}
\end{align}
These identities allow us to write down the induced symmetries of $\mathcal{C}_{30}$ on $\NOP$ in terms of $R(\alpha),X({\theta, \varphi}), M(\theta, \phi)$. For doing so, when $M \in \R^{k \times 3}$ (for $k \in \{2,3\}$) is a matrix and $\PPP \subseteq \R^{3}$ a polyhedron (i.e. a set of points in $\R^3$), we will use the natural notation $M \cdot \PPP$ for the set of points in $\R^k$ given by $\{M \cdot P \text{ for } P \in \PPP\}$. In particular, $M(\theta, \phi) \cdot \PPP$ is the polygon given by the projection of $\PPP$ in direction~$X({\theta, \varphi})$.

\begin{lem} \label{lem:symmetries}
Let $\PPP = \NOP$, then for all $\theta, \varphi, \alpha \in \R$, the following three identities hold (as sets):
\begin{align*}
    M({\theta+2\pi/15,\varphi})\cdot \PPP &=M(\theta, \phi) \cdot \PPP,\\
    R(\alpha+\pi)M(\theta, \phi) \cdot \PPP &=R(\alpha)M(\theta, \phi) \cdot \PPP,\\
    \begin{pmatrix} 
        1&0\\
        0&-1
    \end{pmatrix}
    M(\theta, \phi) \cdot \PPP&=
    M({\theta+\pi/15,\pi-\varphi}) \cdot \PPP.
\end{align*}
\end{lem}

\begin{proof}
    The first two statements follow directly from the definition of $\mathcal{C}_{30}$, the identity (\ref{eq:XMRfromR2}) above and the pointsymmetry of $\PPP$. The third identity can be checked with a simple computation using~that
    \[
        \begin{pmatrix}
            1 & 0\\
            0 & -1
        \end{pmatrix} \cdot 
        M(\theta, \phi) = - M(\theta + \pi/15, \pi - \phi) \cdot R_z(16 \pi/15),
    \]
    and that $- R_z(16 \pi/15) P \in \PPP$ if $P \in \PPP$.
\end{proof}

Recall from \cref{def:rupert} that a pointsymmetric polyhedron $\PPP$ has Rupert's property if there exist $\theta_1,\theta_2\in[0,2\pi)$, $\varphi_1,\varphi_2\in[0,\pi]$ and $\alpha \in [-\pi, \pi)$ such that Rupert's condition (\ref{eq:rupert_condition}) is satisfied. Now the symmetries detailed in \cref{lem:symmetries} allow us to drastically reduce the space of $\theta_1,\theta_2,\varphi_1,\varphi_2, \alpha$:

\begin{corr} \label{cor:NOPbounds}
If $\NOP$ is Rupert, then there exists a solution to its Rupert's condition (\ref{eq:rupert_condition}) with  
\begin{align*}    
\theta_1,\theta_2&\in[0,2\pi/15] \subset [0,0.42], \\
\varphi_1&\in [0,\pi] \subset [0,3.15],\\
\varphi_2&\in [0,\pi/2] \subset [0,1.58],\\
\alpha &\in [-\pi/2,\pi/2] \subset [-1.58,1.58].
\end{align*}
\end{corr}

\begin{proof}
    Let us assume a solution $(\theta_1,\varphi_1,\theta_2,\varphi_2,\alpha)$ to the Rupert condition
    $R(\alpha) M(\theta_1, \phi_1)\PPP \subset (M(\theta_2, \phi_2)\PPP)^\circ$
    is given. Then, this inclusion still holds true, when mirroring all the points along the $x$-axis, hence
    $$\begin{pmatrix} 
        1&0\\
        0&-1
    \end{pmatrix}
    R(\alpha) M(\theta_1, \phi_1)\PPP \subset \begin{pmatrix} 
        1&0\\
        0&-1
    \end{pmatrix}(M(\theta_2, \phi_2)\PPP)^\circ.$$
    Using 
    \[
        \begin{pmatrix} 
            1&0\\
            0&-1
        \end{pmatrix}R(\alpha)=
        R({-\alpha})\begin{pmatrix} 
            1&0\\
            0&-1
        \end{pmatrix}
    \]
    and the third identity in \cref{lem:symmetries}, we arrive at
    \[
        R({-\alpha})M({\theta_1+\pi/15,\pi-\varphi_1}) \PPP
        \subset (M({\theta_2+\pi/15,\pi-\varphi_2})\PPP)^\circ.
    \]
    So, if $(\theta_1,\varphi_1,\theta_2,\varphi_2,\alpha)$ is a solution to Rupert's condition for $\PPP$, then 
    \[
        (\theta_1+\pi/15,\pi-\varphi_1,\theta_2+\pi/15,\pi-\varphi_2,-\alpha)
    \]
    is also a solution. It follows that we may assume that there exists a solution with $\varphi_2\in[0,\pi/2]$. Finally, using the first and second identity of \cref{lem:symmetries}, we can enforce the remaining parameters to lie in the claimed intervals.
\end{proof}

\section{Bounding $R_a(\alpha), R(\alpha), M(\theta, \phi)$} \label{sec:bounds}

For a matrix $A\in \mathbb{R}^{m \times n}$ we will write $\|A\|$ for its (operator) norm given by $\|A\| = \max_{v\in \R^n}\|Av\|/\|v\|$, where $\|v\|$ for vectors $v \in \mathbb{R}^n$ denotes the usual Euclidean norm, as before. Recall that $\|.\|$ is sub-multiplicative and that multiplication by an orthogonal matrix (a matrix $Q \in \R^{n \times n}$, such that $QQ^t = \id$) leaves the norm unchanged. 


\begin{lem} \label{lem:RaRalpha}
For any $\alpha, \theta,\varphi \in \R$ and $a \in \{x,y,z\}$ one has $\| R(\alpha)\| = \| R_a(\alpha)\| =\| M(\theta, \phi)\| = 1$.
\end{lem}
\begin{proof}
The statement is obvious for $R(\alpha), R_x(\alpha), R_y(\alpha), R_z(\alpha)$ since these matrices describe rotations. Identity (\ref{eq:XMRfromR2}) shows that $\| M(\theta, \phi)\| \leq 1$ by sub-multiplicativity, and it is easy to check that for the unit vector $v = (-\sin(\theta), \cos(\theta),0)^t$ it holds that $\|M(\theta, \phi) v\| = 1$.
\end{proof}

The following easy lemma bounds $\|R_a(\alpha)-R_a({\alphab})\|$ and $\|R(\alpha)-R(\alphab)\|$:

\begin{lem} \label{lem:RaRa}
Let $\epsilon>0$, $|\alpha-\alphab|\leq\varepsilon$ and $a \in \{x,y,z\}$ then
$\|R_a(\alpha)-R_a({\alphab})\|=\|R(\alpha)-R(\alphab)\| < \varepsilon$. 
\end{lem}

\begin{proof}
    We show the equality for $a=x$ only; so we calculate 
    \begin{align*}
        \|R_x(\alpha)-R_x({\alphab})\| &= \max_{v_1^2+v_2^2+v_3^2=1}\left\|
    \begin{pmatrix} 
        0 & 0 & 0\\ 
        0 & \cos(\alpha)-\cos(\alphab) & -\sin(\alpha)+\sin(\alphab)\\
        0 & \sin(\alpha)-\sin(\alphab) & \cos(\alpha)-\cos(\alphab)
    \end{pmatrix} 
    \begin{pmatrix} 
        v_1\\
        v_2\\
        v_3
    \end{pmatrix}\right\|\\
    &=\max_{v_2^2+v_3^2=1}\left\|\left(R(\alpha)-R(\alphab)\right)\begin{pmatrix} 
        v_2\\
        v_3
    \end{pmatrix}\right\|\\
    &=\|R(\alpha)-R(\alphab)\|.
    \end{align*}
    For the inequality we first observe that for any $v\in\R^2$ the points $R(\alpha)v$ and $R(\alphab)v$ lie on the circle with radius $\|v\|$. The length of the arc connecting them has length at most $|\alpha-\alphab|\cdot \|v\|$ and is always longer than the straight line connecting them with length $\|R(\alpha)v-R(\alphab)v\|$. Hence we~have  
    \[
        \|R(\alpha)-R(\alphab)\|=\max_{v\in\R^2} \frac{\|(R(\alpha)-R(\alphab))v\|}{\|v\|} < \max_{v\in\R^2} \frac{|\alpha-\alphab|\cdot \|v\|}{\|v\|}=|\alpha-\alphab|\leq \varepsilon. \qedhere
    \]
\end{proof}

In the same way we would like to bound $\|M(\theta, \phi)-M(\thetab,\phib)\|$, $\|X({\theta, \varphi})-X(\thetab,\phib)\|$ as well as $\|R(\alpha) M(\theta, \phi)-R(\alphab)M(\thetab,\phib)\|$ for  $|\theta-\thetab|,|\varphi-\phib|, |\alpha-\alphab|\leq\varepsilon$. It turns out that having sharp bounds for these terms is more complicated and, amongst others ideas, this relies on the following sharp inequality which is itself a consequence of the convexity (for small $x$) of $\cos(\sqrt{x})$ and Jensen's inequality.

\begin{lem} \label{lem:jensen}
    For all $a,b \in \R$ with $|a|,|b|\leq 2$ the following inequality holds:
    \[
        (1+\cos(a))(1+\cos(b))\geq 2+2\cos\Big(\sqrt{a^2+b^2}\Big),
    \]
    with equality only for $a=0$ or $b=0$.
\end{lem}

\begin{proof}
    Using $1+ \cos(x) = 2 \cos(x/2)^2$ three times the inequality becomes equivalent to 
    \[
        \cos(a/2)^2 \cos(b/2)^2 \geq \cos(\sqrt{a^2+b^2}/2)^2.
    \]
    Note that without loss of generality we can assume that $0 \leq a, b \leq 2$; let $x=a/2$ and $y=b/2$ so that $0 \leq x,y\leq 1$. Then we need to prove that 
    \[
        \cos(x) \cos(y) \geq \cos(\sqrt{x^2+y^2}),
    \]
    Or, equivalently, 
    \[
        \frac{\cos(x+y) + \cos(x-y)}{2} \geq \cos(\sqrt{x^2+y^2}),
    \]
    as $2 \cos(x) \cos(y) = \cos(x+y) + \cos(x-y)$.
    Now the crucial observation is that $f(t) \coloneqq \cos(\sqrt{t})$ is \emph{convex} for $t \in [0,4]$. Indeed, 
    \[
        f''(t) = \frac{\sin(\sqrt{t}) - \sqrt{t} \cos(\sqrt{t})}{4 \sqrt{t^3}},
    \]
    and $f''(t) \geq 0$ on $t \in [0,4]$ is equivalent to $\sin(s) - s \cos(s) \geq 0$ on $s \in [0,2]$. 
    
    We conclude using Jensen's inequality:
    \begin{align*}
        \frac{\cos(x+y) + \cos(x-y)}{2} = & \frac{f((x+y)^2)+f((x-y)^2)}{2} \\
        & \geq f\left( \frac{(x+y)^2+(x-y)^2}{2} \right) = \cos(\sqrt{x^2+y^2}).
    \end{align*}
    Equality can only happen for $(x + y)^2 = (x - y)^2$, i.e. $xy=0$.
\end{proof}

\begin{lem} \label{lem:RxRy}
For any $\alpha,\beta\in \mathbb{R}$ one has 
\[
    \|R_x(\alpha)R_y(\beta)-\id\| \leq  \sqrt{\alpha^2+\beta^2}
\]
with equality only for $\alpha = \beta = 0$.
\end{lem}
\begin{proof}
    Without loss of generality we can assume $|\alpha|,|\beta|\leq 2$, since if the equation is true for $(\alpha,\beta)$, then it is also true for $(2\alpha,2\beta)$ as
    \begin{align*}
    \|R_x(2\alpha)R_y(2\beta)-\id\| &\leq \|R_x(2\alpha)R_y(2\beta)-R_x(\alpha)R_y(\beta)\|+\|R_x(\alpha)R_y(\beta)-\id\|\\
    &=2\|R_x(\alpha)R_y(\beta)-\id\|,
    \end{align*}
    using the triangle inequality and multiplying by $R_x(-\alpha)$ from the left and $R_y(-\beta)$ from the right. Define $T\coloneqq R_x(\alpha)R_y(\beta)$. As $T$ is the composition of two rotation matrices, it is itself a rotation matrix. Therefore $T$ is of the form 
    \[
        T= UR_x(\Psi)U^{-1}
    \]
    for some orthogonal matrix $U$ and $-\pi \leq \Psi \leq \pi$. We therefore have
    \begin{align*}
    &\text{tr}(R_x(\alpha)R_y(\beta))=\text{tr}(T)=\text{tr}(UR_x(\Psi)U^{-1})=\text{tr}(R_x(\Psi)).
    \end{align*}
It follows that
\[
    \cos(\alpha)+\cos(\beta)+\cos(\alpha)\cos(\beta)=1+2\cos(\Psi),
\]
which is trivially equivalent to $(1+\cos(\alpha))(1+\cos(\beta)) =  2+2\cos(\Psi)$.
As $|\alpha|,|\beta| \leq 2$, we can bound the left-hand-side from below using \cref{lem:jensen}: 
\begin{align*}
    2+2\cos\Big(\sqrt{\alpha^2+\beta^2}\Big)&\leq 2+2\cos(\Psi),
\end{align*}
thus $\cos\Big(\sqrt{\alpha^2+\beta^2}\Big) \leq \cos(\Psi)$. Since $\sqrt{2\cdot 2^2} = \sqrt{8} < 3$ the arguments of the cosines are between $-\pi$ and $\pi$ and we can conclude that
\(
    \sqrt{\alpha^2+\beta^2}\geq |\Psi|.
\)
Combining everything we find
\begin{align*}
  \|R_x(\alpha)R_y(\beta)-\text{Id}\| =\|UR_x(\Psi)U^{-1}-\text{Id}\| =\|R_x(\Psi)-\text{Id}\| =\|R_x(\Psi)-R_x(0)\|
    \leq |\Psi|\leq \sqrt{\alpha^2+\beta^2}
\end{align*}
using \cref{lem:RaRa} in the fourth step. \cref{lem:RaRa} also infers that the inequality in this step is strict except for $\Psi = 0$. In this case, however, the inequality in step five is strict except for $\alpha = \beta = 0$.

\end{proof}

\begin{lem} \label{lem:sqrt2}
Let $\epsilon>0$ and $|\theta-\thetab|,|\varphi-\phib| \leq \varepsilon$ then 
$\|M(\theta, \phi)-M(\thetab,\phib)\|, \|X({\theta, \varphi})-X(\thetab,\phib)\| < \sqrt{2}\varepsilon.$
\end{lem}

\begin{proof}
    For $(\theta, \phi) = (\thetab, \phib)$ the result is trivial since $\epsilon>0$. In the other case,
    recall the identity (\ref{eq:XMRfromR2})
    \[
    M(\theta, \phi) =
        \begin{pmatrix} 
            0 & 1 & 0\\ 
            -1 & 0 & 0\\
        \end{pmatrix}
        \cdot R_y(\varphi) \cdot R_z(-\theta),
    \]
    and denote by $Z$ the $2 \times 3$ matrix on the rind-hand side.
    It follows that
    \begin{align*}
    \|M(\theta, \phi)-M(\thetab,\phib)\| \leq \|Z\|\cdot \|R_y(\varphi) \cdot R_z(-\theta)- R_y(\phib) \cdot R_z(-\thetab)\|
    \end{align*}
    from the sub-multiplicativity of the norm. Using $\|Z\|=1$ and that $R_y(-\varphi), R_z(\theta)$ are orthogonal we obtain
\begin{align*}
    \|M(\theta, \phi)-M(\thetab,\phib)\| \leq \|\text{Id}-R_y(\phib-\varphi)R_z(\theta-\thetab)\|. 
\end{align*}
The right-hand side is strictly smaller than $\sqrt{(\theta-\thetab)^2 + (\varphi-\phib)^2} \leq \sqrt{2}\epsilon$ by \cref{lem:RxRy} using that $(\theta, \phi) \neq (\thetab, \phib)$. 

    The proof for $\|X({\theta, \varphi})-X(\thetab,\phib)\|$ is analogous using the corresponding identity (\ref{eq:XMRfromR1})    for $X({\theta, \varphi})$.
\end{proof}

The following two statements are easy applications of this lemma:
\begin{lem} \label{lem:XP>0}
    Let $P \in \R^3$ with $\|P\| \leq 1$. Further, let $\epsilon>0$ and $\thetab,\phib, \theta, \phi \in \R$ such that $|\thetab-\theta|, |\phib - \phi| \leq \epsilon$. If  
    \(
        \langle X(\thetab,\phib),P \rangle>\sqrt{2}\varepsilon
    \)
    then
    \(
        \langle X(\theta, \phi),P \rangle>0.
    \)
\end{lem}
\begin{proof}
    It follows from the Cauchy-Schwarz inequality that 
    \begin{align*}
     \langle X(\theta, \phi), P \rangle \geq \langle X(\thetab,\phib), P \rangle - \| X(\thetab,\phib) - X(\theta, \phi) \| \cdot \| P \|.
    \end{align*}
    The right-hand side is larger than $\sqrt{2}\epsilon - \sqrt{2}\epsilon \cdot 1 = 0$ by \cref{lem:sqrt2} and using that $\|P\| \leq 1$.
\end{proof}
\begin{lem} \label{lem:MP>r}
    Let $P \in \R^3$ with $\|P\| \leq 1$. Further, let $\epsilon, r>0$ and $\thetab,\phib, \theta, \phi \in \R$ such that $|\thetab-\theta|, |\phib - \phi| \leq \epsilon$. If  
    \(
        \| M(\thetab,\phib) P \| > r + \sqrt{2}\varepsilon
    \)
    then
    \(
        \| M(\theta,\phi) P \| > r.
    \)
\end{lem}
\begin{proof}
    It follows from the triangle inequality that 
    \[
        \| M(\theta, \phi) P \| \geq \| M(\thetab, \phib) P \| - \| M(\theta, \phi) - M(\thetab, \phib)\| \|P\|,
    \]
    then the right-hand side is larger than $r + \sqrt{2} \epsilon - \sqrt{2}\epsilon \cdot 1 = r$ by \cref{lem:sqrt2} and using that $\|P\| \leq 1$.
\end{proof}

Now let us prove a version of \cref{lem:sqrt2} when a rotation $R(\alpha)$ is also involved:

\begin{lem} \label{lem:sqrt5}
    Let $\epsilon>0$ and $|\theta-\thetab|,|\varphi-\phib|,|\alpha-\alphab|\leq\varepsilon$ then
    $\|R(\alpha) M(\theta, \phi)-R(\alphab)M(\thetab,\phib)\| < \sqrt{5} \varepsilon.$
\end{lem}

\begin{proof}
    Recall the identity (\ref{eq:XMRfromR3})
    \[
    R(\alpha) M(\theta, \phi)=
       \begin{pmatrix} 
            0 & 1 & 0\\ 
            -1 & 0 & 0\\
        \end{pmatrix}
        R_z(\alpha) \cdot R_y(\varphi) \cdot R_z(-\theta)
    \] 
    to obtain like in the proof of \cref{lem:sqrt2}:
    \begin{align*}
        \|R(\alpha) M(\theta, \phi)-R(\alphab)M(\thetab,\phib)\| \leq \|R_z(\alpha-\alphab) \cdot R_y(\varphi) - R_y(\phib) \cdot R_z(\theta-\thetab)\|
    \end{align*}
    using the sub-multiplicativity of the norm.

    Now define   
    \[
        \Phi = \frac{\varphi\cdot |\theta-\thetab| +\phib\cdot |\alpha-\alphab|}{|\alpha-\alphab|+|\theta-\thetab|},
    \]
    and obtain further using the triangle inequality and \cref{lem:RxRy}:
    \begin{align*}
        \|R(\alpha) M(\theta, \phi)-R(\alphab)M(\thetab,\phib)\| 
        &\leq \|R_z(\alpha-\alphab) \cdot R_y(\varphi) - R_y(\Phi)\|+ \|R_y(\Phi)-R_y(\phib) \cdot R_z(\theta-\thetab)\|\\
        &= \|R_z(\alpha-\alphab) \cdot R_y(\varphi-\Phi) - \text{Id}\|+ \|\text{Id}-R_y(\phib-\Phi) \cdot R_z(\theta-\thetab)\|\\
        &\leq \sqrt{(\alpha-\alphab)^2+(\varphi-\Phi)^2}+\sqrt{(\theta-\thetab)^2+(\phib-\Phi)^2}.
    \end{align*}
    Note that also by \cref{lem:RxRy} this inequality is strict except if both sides vanish.
    The definition of $\Phi$ conveniently implies after a small computation that 
    \[
        \sqrt{(\alpha-\alphab)^2+(\varphi-\Phi)^2}+\sqrt{(\theta-\thetab)^2+(\phib-\Phi)^2} = \sqrt{\left(|\alpha - \alphab|+|\theta - \thetab|\right)^2 + |\phi-\phib|^2} \leq \sqrt{5}\varepsilon,
    \]
    which concludes the proof.
\end{proof}

\section{The global theorem} \label{sec:global_thm}

Recall from \cref{def:rupert} the definition of Rupert's property and the corresponding condition (\ref{eq:rupert_condition})
\[
    R(\alpha) M(\theta_1, \phi_1)\PPP \subset (M(\theta_2, \phi_2)\PPP)^\circ.
\]
It is clear that if $\Psi = (\theta_1,\varphi_1,\theta_2,\varphi_2,\alpha) \in \R^5$ is a solution to Rupert's problem then in particular, since~(\ref{eq:rupert_condition}) is satisfied, for every vertex $S \in \PPP$ we must have that $R(\alpha) M(\theta_1, \phi_1)S$ lies inside the convex hull of $M(\theta_2, \phi_2)\PPP$. This, in turn, intuitively implies that for any vector $w \in \R^2$ it holds~that
\[
    \langle R(\alpha) M(\theta_1, \phi_1)S,w \rangle \leq \max_{P \in \PPP} \langle M(\theta_2, \phi_2)P ,w\rangle.
\]
We prove this intuitive fact in \cref{lem:hullscalarprod}. Simply rephrasing this observation, we can say that if there exists a vertex $S \in \PPP$ and a vector $w \in \R^2$ such that 
\[
    \langle R(\alpha) M(\theta_1, \phi_1)S ,w\rangle > \max_{P \in \PPP} \langle M(\theta_2, \phi_2)P ,w\rangle,
\]
then $(\theta_1,\varphi_1,\theta_2,\varphi_2,\alpha) \in \R^5$ cannot be a solution to Rupert's problem. This basic idea is crucial, as it allows us not only to reject $(\theta_1,\varphi_1,\theta_2,\varphi_2,\alpha)$ as a possible solution but, often in practice, also a whole \emph{region} close to this point.  For instance, a direct application of \cref{lem:sqrt2} and \cref{lem:sqrt5} shows that if 
\[
    \langle R(\alphab) M(\thetab_1, \phib_1)S,w \rangle > \max_{P \in \PPP} \langle M(\thetab_2, \phib_2)P,w \rangle + (\sqrt{2} + \sqrt{5}) \rho \epsilon
\]
for some pointsymmetric convex $\PPP$ of radius $\rho$, some $S \in \PPP$ and $w \in \R^2$ with $\|w\|=1$, then there cannot be a solution $(\theta_1,\varphi_1,\theta_2,\varphi_2,\alpha) \in [\thetab_1\pm\epsilon,\phib_1\pm\epsilon,\thetab_2\pm\epsilon,\phib_2\pm\epsilon,\alphab\pm\epsilon] $. We will prove a stronger version taking into account the derivatives $M^\theta, M^\phi, R'$ (see \cref{sec:derivatives} for precise definitions) at the point $(\thetab_1,\phib_1,\thetab_2,\phib_2,\alphab)$:

\begin{thm}[Global Theorem] \label{thm:global}
    Let $\PPP$ be a pointsymmetric convex polyhedron with radius $\rho =1$ and let $S \in \PPP$. Further let $\thetab_1,\phib_1,\thetab_2,\phib_2,\alphab \in \R$ and let $w\in\R^2$ be a unit vector. Denote $\Mib \coloneqq M(\thetab_1, \phib_1)$, $ \Miib \coloneqq M(\thetab_2, \phib_2)$ as well as $\Mib^{\theta} \coloneqq M^\theta(\thetab_1, \phib_1)$, $\Mib^{\phi} \coloneqq M^\phi(\thetab_1, \phib_1)$ and analogously for $\Miib^{\theta}, \Miib^{\phi}$. Finally set
    \begin{align*}
        G& \coloneqq \langle R(\alphab) \Mib S,w \rangle - \epsilon\cdot\big(|\langle R'(\alphab)  \Mib S,w \rangle|+|\langle R(\alphab) \Mib^\theta S,w \rangle|+|\langle R(\alphab) \Mib^\phi S,w \rangle|\big)- 9\epsilon^2/2,\\
        H_P & \coloneqq \langle \Miib P,w \rangle + \epsilon\cdot\big(|\langle \Miib^\theta P,w \rangle|+|\langle  \Miib^\varphi P,w \rangle|\big) + 2\epsilon^2, \quad \text{ for } P \in \PPP.
    \end{align*}
    If $G>\max_{P\in \PPP} H_P$ then there does not exist a solution to Rupert's condition (\ref{eq:rupert_condition}) with 
    $$(\theta_1,\varphi_1,\theta_2,\varphi_2,\alpha) \in U \coloneqq [\thetab_1\pm\epsilon,\phib_1\pm\epsilon,\thetab_2\pm\epsilon,\phib_2\pm\epsilon,\alphab\pm\epsilon] \subseteq \R^5.$$
\end{thm}

\begin{example}
    Consider the Octahedron $\OOO = \{ (\pm 1, 0,0), (0, \pm1,0),(0,0,\pm1) \} \subseteq \R^3$ and two projection directions $(\thetab_1, \phib_1) = (0,0)$ and $(\thetab_2, \phib_2) = (\pi/4,\tan^{-1}(\sqrt{2}))$. 
    We wish to apply \cref{thm:global} to show that with $\alphab=0$ this choice of parameters does not yield a solution to Rupert's problem for $\OOO$ and, moreover, that there exists $\epsilon>0$ such that there is also no solution $(\theta_1,\varphi_1,\theta_2,\varphi_2,\alpha)$ with $|\thetab_i - \theta_i|,|\phib_i - \phi_i|, |\alpha| \leq \epsilon$ either. 
    In other words, the red projection in Figure~\ref{fig:octahedron_glob} is not inside the black one, even after perturbing the projection parameters by some small $\epsilon>0$.
    
    \begin{figure}[htb]
        \centering
        \includegraphics[width=1\linewidth]{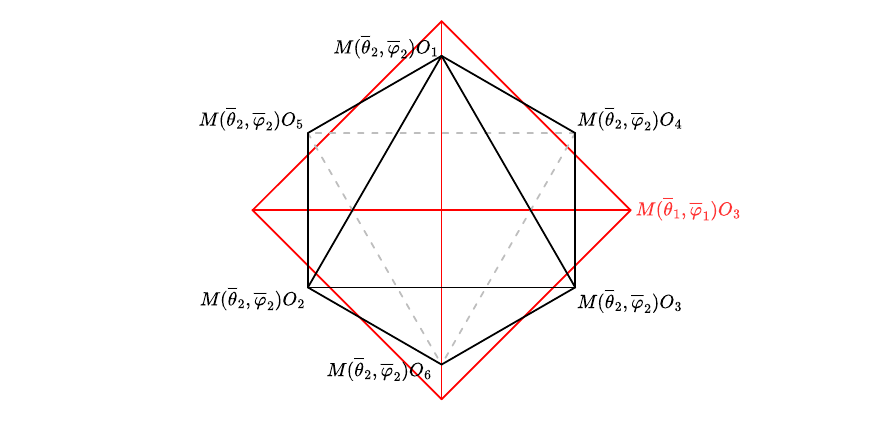}
        \caption{Two projections of the Octahedron: one in direction $(\thetab_1, \phib_1) = (0,0)$ (in red) and one in direction $(\thetab_2, \phib_2) = (\pi/4, \tan^{-1}(\sqrt{2})$ (in black).}
        \label{fig:octahedron_glob}
    \end{figure}
    
    Following \cref{thm:global} we first compute:
    \[
     \Mib =  \begin{pmatrix}
    0 & 1 & 0 \\
    -1 & 0 & 0
\end{pmatrix}  \quad \text{ and } \quad \Miib = \begin{pmatrix}
    -\sqrt{2}/2 & \sqrt{2}/2 & 0 \\
    -\sqrt{6}/6 & -\sqrt{6}/6 & \sqrt{6}/3
\end{pmatrix}.
    \]
    Moreover, for the derivatives we find
    \[
        \Mib^\theta =  
\begin{pmatrix}
    -1 & 0 & 0 \\
    0 & -1 & 0
\end{pmatrix}, \quad \Mib^\phi = 
\begin{pmatrix}
    0 & 0 & 0 \\
    0 & 0 & 1
\end{pmatrix}, \quad \text{ as well as}
    \]
    \[
        \Miib^\theta =  
\begin{pmatrix}
    -\sqrt{2}/2 & -\sqrt{2}/2 & 0 \\
    \sqrt{6}/6 & -\sqrt{6}/6 & 0
\end{pmatrix}  \quad \text{ and } \quad \Miib^\phi = 
\begin{pmatrix}
    0 & 0 & 0 \\
    \sqrt{3}/3 & \sqrt{3}/3 & \sqrt{3}/3
\end{pmatrix}.
    \]
    Further note that $R(0) = \id$ and $R'(0) = R(\pi/2)$. We choose $S = O_3 = (0,1,0)$ the preimage of the vertex on the very right in Figure~\ref{fig:octahedron_glob} and $w = (1,0)$ the direction of $M(\thetab_1,\phib_1) S$.
    
    After a straightforward calculation we obtain 
    \[
        G = 1 - 9\epsilon^2/2
    \]
    and
    \[
        H_P = 
        \begin{cases}
            2 \epsilon ^2 \quad & \text{ for } P = M(\thetab_2, \phib_2) O_i \; \text{ with } \; i=1,6, \\
            \sqrt{2}/2 + \epsilon \sqrt{2}/2 + 2 \epsilon^2 \quad & \text{ for } P = M(\thetab_2, \phib_2) O_i \; \text{ with } \; i=3,4,\\
            -\sqrt{2}/2 + \epsilon \sqrt{2}/2 + 2 \epsilon^2 \quad & \text{ for } P = M(\thetab_2, \phib_2) O_i \; \text{ with } \; i=2,5.
        \end{cases}
    \]
    It is then obvious that $\max H_P$ is $\sqrt{2}/2 + \epsilon \sqrt{2}/2 + 2 \epsilon^2$ and that $G = 1 - 9\epsilon^2/2 > \max H_P$ for small $\epsilon$. Specifically, if $\epsilon < 0.164$ then \cref{thm:global} can be applied and a solution in the region described above can be ruled out.

    We note that the strategy of the proof below --  applied to the current example -- is to show that the most right point in Figure~\ref{fig:octahedron_glob} cannot move inside the black projection if the parameters $\thetab_1, \phib_1,\thetab_2, \phib_2,\alphab$ are changed by at most $\epsilon$. 
\end{example}

\subsection{Definitions and lemmas} \label{sec:derivatives}

\begin{lem} \label{lem:hullscalarprod}
    If $\mathcal{P} \subseteq \R^2$ is a convex polygon, $S \in \mathcal{P}^\circ$ and $w \in \R^2$ then
    \[
        \langle S ,w \rangle \leq \max_{P \in \PP} \langle P ,w\rangle.
    \]
\end{lem}

\begin{proof}
    Write $S = \sum_i \lambda_i P_i$ for $P_i \in \PP$ the vertices of $\PP$ and $\lambda_i \in [0,1)$ such that $\sum_i \lambda_i = 1$. Moreover, denote $M \coloneqq \max_{P \in \PP} \langle P ,w\rangle $. Then:
    \[
        \langle S ,w \rangle = \sum_i \lambda_i \langle P_i, w \rangle \leq M \cdot \sum_i \lambda_i = M. \qedhere
    \]
\end{proof}

As explained above, we will use derivatives of the matrices encountered so far. Naturally, we write, for example,
\begin{align*}
    M^\theta(\theta, \phi)& \coloneqq \frac{\dd{}}{\dd{}\theta} M(\theta, \phi)  =
    \begin{pmatrix} 
    -\cos(\theta) & -\sin(\theta) & 0 \\ 
    \sin(\theta)\cos(\phi) & -\cos(\theta)\cos(\phi) & 0
    \end{pmatrix} \quad \text{and}\\
    M^\phi(\theta, \phi) & \coloneqq \frac{\dd{}}{\dd{}\phi} M(\theta, \phi) =
    \begin{pmatrix} 
    0 & 0 & 0 \\ 
    \cos(\theta)\sin(\phi) & \sin(\theta)\sin(\phi) & \cos(\phi)
    \end{pmatrix}.
\end{align*}
Note that this, for instance, means that 
\[
    M^\phi(x, y) =
    \begin{pmatrix} 
    0 & 0 & 0 \\ 
    \cos(x)\sin(y) & \sin(x)\sin(y) & \cos(y)
    \end{pmatrix}.
\]
Moreover, in the same way we also write for the derivative of the rotation matrix in two dimensions:
\[
R'(\alpha) \coloneqq \frac{\dd{}}{\dd{}\alpha} R(\alpha) = 
    \begin{pmatrix} 
    -\sin(\alpha) & -\cos(\alpha) \\ 
    \cos(\alpha)& -\sin(\alpha)\\ 
    \end{pmatrix}
\]
and for the derivatives of the three-dimensional rotations along axes we analogously use:
\begin{align*}
R_{x}'(\alpha)\coloneqq 
    \begin{pmatrix} 
    0 & 0 & 0\\ 
    0 & -\sin(\alpha) & -\cos(\alpha)\\
    0 & \cos(\alpha) & -\sin(\alpha)
    \end{pmatrix}, & \qquad R_{y}'(\alpha)\coloneqq 
    \begin{pmatrix} 
    -\sin(\alpha) & 0 & -\cos(\alpha)\\
    0 & 0 & 0\\ 
    \cos(\alpha) & 0 & -\sin(\alpha) 
    \end{pmatrix} \quad \text{and} \\
    R_{z}'(\alpha)\coloneqq &
    \begin{pmatrix} 
    -\sin(\alpha) & -\cos(\alpha) & 0 \\
    \cos(\alpha) & -\sin(\alpha) & 0\\
    0 & 0 & 0
    \end{pmatrix}.
\end{align*}
Similarly we shall denote the second derivatives by $M^{\theta \theta}(\theta, \phi), M^{\theta \phi}(\theta, \phi), R''(\alpha), R_{x}''(\alpha)$, etc.

\bigskip

It is easy to check that all derivatives of the rotation matrices have unit norm, e.g., $\|R_{z}''(\alpha)\|= 1$. Moreover, the following identities follow from the chain rule and (\ref{eq:XMRfromR2}):
\begin{align*}    
    M^\theta(\theta,\varphi) &=-
    \begin{pmatrix} 
        0 & 1 & 0\\ 
        -1 & 0 & 0\\
    \end{pmatrix}
    \cdot R_y(\varphi) \cdot R_z'(-\theta),\\
    M^\varphi (\theta,\varphi) &=
    \begin{pmatrix} 
        0 & 1 & 0\\ 
        -1 & 0 & 0\\
    \end{pmatrix}
    \cdot R_{y}'(\varphi) \cdot R_z(-\theta).
\end{align*}

\begin{lem} \label{lem:leq1}
    Let $S \in \R^3$ and $w \in \R^2$ be unit vectors and set $f(x_1,x_2,x_3) = \langle R(x_3) M(x_1,x_2)S,w \rangle$. Then for all $x_1,x_2,x_3 \in \R$ and any $i,j \in \{1,2,3\}$ it holds that
    \[
        \left|\frac{\dd^2 f}{\dd x_i \dd x_j}(x_1,x_2,x_3)\right|\leq 1.
    \]
\end{lem}

\begin{proof}
We shall prove the statement for $(i,j)=(1,2)$ only as the other cases are analogous. First we rewrite $f(x_1,x_2,x_3)$ as
\[
    f(x_1,x_2,x_3) = \left\langle 
        \begin{pmatrix} 
            0 & 1 & 0\\ 
            -1 & 0 & 0\\
        \end{pmatrix} R_z(x_3) \cdot R_y(x_2) \cdot R_z(-x_1)S,w \right\rangle,
\]
then it follows from the chain rule that 
    \[
        \frac{\dd^2 f}{\dd x_1 \dd x_2}(x_1,x_2,x_3) = -\left\langle 
        \begin{pmatrix} 
            0 & 1 & 0\\ 
            -1 & 0 & 0\\
        \end{pmatrix} R_z(x_3) \cdot R_y'(x_2) \cdot R_z'(-x_1)S,w \right\rangle,
    \]
then the statement follows from the Cauchy-Schwarz inequality, the bounds on the matrix norms and $\|S\| = \|w\| = 1$.
\end{proof}

In the proof we will also use the following fact in which we write $\partial_{x_i}f(x_1,\dots,x_n)$ for the derivative of some differentiable $f: \R^n \to \R$ with respect to $x_i$:

\begin{lem} \label{lem:n2}
    Let $f:\R^n\to \R$ be a $C^2$-function and $x_1,\dots,x_n,y_1,\dots,y_n \in \R$ such that $|x_1-y_1|,\dots,|x_n-y_n|\leq \varepsilon$. 
    If 
    \(
        \left|\partial_{x_i}\partial_{x_j}f(x_1,\dots,x_n)\right| \leq 1
    \)
    for all $i,j \in \{1,\dots,n\}$ then
    \[
    |f(x_1,\dots,x_n)-f(y_1,\dots,y_n)|\leq \varepsilon \sum_{i=1}^n |\partial_{x_i} f(x_1,\dots,x_n)| + \frac{n^2}{2}\varepsilon^2.
    \]
\end{lem}
Note that the factor $n^2/2$ in this lemma is optimal as for the function $f(x_1,\dots,x_n)=(x_1+\cdots + x_n)^2/2$ at $x_1=\cdots=x_n=0$ and $y_1=\cdots=y_n=\varepsilon$ equality holds.

\begin{proof}
    Let $x = (x_1,\dots,x_n)$ and $y = (y_1,\dots,y_n)$. We define $g:[0,1]\to \R$ by
    \[
        g(t) = f\big(x_1(1-t)+y_1t, \dots ,x_n(1-t)+y_nt\big) = f(x(1-t)+yt)
    \]
    and observe that 
    \[
        g(1)-g(0)=\int_0^1g'(t)\text{d}t=\int_0^1\int_0^t g''(s)\text{d}s+g'(0)\text{d}t=g'(0)+\int_0^1\int_0^t g''(s)\text{d}s\text{d}t.
    \]
    Note that $|g(1)-g(0)|$ is exactly the term which needs to be bounded.
    The chain rule implies that 
    \[
        g'(t)=\sum_{i=1}^n (y_i-x_i) \partial_{x_i} f(x(1-t)+yt),
    \]
    thus at $t=0$ we find
    \begin{align*}
    |g'(0)| \leq\sum_{i=1}^n |y_i-x_i| \cdot | \partial_{x_i} f(x_1,\dots,x_n)| \leq \epsilon \sum_{i=1}^n |\partial_{x_i} f(x_1,\dots,x_n)|. 
    \end{align*}
    For the second derivative of $g(t)$ we also get with the chain rule
    $$g''(t)=\sum_{i=1}^n \sum_{j=1}^n(y_i-x_i)(y_j-x_j) \cdot \partial_{x_i} \partial_{x_j} f\big(x(1-t)+y t\big),$$
   which implies that
    $|g''(t)|\leq n^2 \varepsilon^2$.
    Altogether, one obtains
    \begin{align*}
      |f(x_1,\dots,x_n)-f(y_1,\dots,y_n)| & = |g(1)-g(0)| \leq |g'(0)|+\int_0^1\int_0^t |g''(s)|\text{d}s\text{d}t\\
      &\leq \epsilon \sum_{i=1}^n | \partial_{x_i} f(x_1,\dots,x_n)| + \frac{n^2}{2}\varepsilon^2. \qquad \qedhere
    \end{align*}
\end{proof}

\subsection{Proof of the global theorem}

\begin{proof}[Proof of \cref{thm:global}]
    We prove the statement by contradiction: Assume that there exists $S \in \PPP$ with a unit vector $w \in \R^2$ like in the assumption and $(\theta_1,\varphi_1,\theta_2,\varphi_2,\alpha) \in U$ such that Rupert's condition is satisfied. As noted before, Rupert's condition and \cref{lem:hullscalarprod} imply in particular that 
    \[
        \max_{P\in \PPP}\langle R(\alpha) M(\theta_1, \phi_1)P,w \rangle < \max_{P\in \PPP}\langle  M(\theta_2, \phi_2)P,w \rangle.
    \]
    We shall, however, show that the following sequence of inequalities yields the desired contradiction: 
    \[
        \max_{P\in \PPP}\langle R(\alpha) M(\theta_1, \phi_1)P,w \rangle \stackrel{\text{(i)}}{\geq} \langle R(\alpha) M(\theta_1, \phi_1)S,w \rangle\stackrel{\text{(ii)}}{\geq} G \stackrel{\text{(iii)}}{>}H\stackrel{\text{(iv)}}{\geq}\max_{P\in \PPP}\langle  M(\theta_2, \phi_2)P,w \rangle.
    \]

    The inequality (i) is obvious and (iii) is an assumption of the theorem. To show inequality (ii) define the function 
    \[
        f(x_1,x_2,x_3) = \langle R({x_3}) M({x_1,x_2})S/\|S\|,w \rangle = \langle R({x_3}) M({x_1,x_2})S,w \rangle/\|S\|.
    \]
    Then \cref{lem:leq1} guarantees that 
    \[
        \left|\frac{\partial^2f}{\partial x_i \partial x_j}(x_1,x_2,x_3)\right|\leq 1,
    \]
    for all $x_1,x_2,x_3 \in \R$ and any $i,j \in \{1,2,3\}$; therefore \cref{lem:n2} applies and implies that
    \[
        |f(\thetab_1,\phib_1,\alphab)-f(\theta_1,\varphi_1,\alpha)|\leq \varepsilon \sum_{i=1}^3 |\partial_{x_i} f(\thetab_1,\phib_1,\alphab)| + \frac{9}{2}\varepsilon^2,
    \]
    as $|\thetab_1-\theta_1|,|\phib_1-\varphi_1|,|\alphab-\alpha|\leq \varepsilon$ by assumption. Multiplying the inequality with $\|S\|$ we get
    \begin{align*}
        |\langle R(a)  M({\thetab_1,\phib_1})S,w \rangle-\langle R(\alpha) M(\theta_1, \phi_1)S,w \rangle| & \\
       & \hspace{-6.5cm} \leq \varepsilon\cdot\bigg(|\langle R(\alpha)  M({\thetab_1,\phib_1})S,w \rangle|+|\langle R(\alphab)  M^\theta({\thetab_1,\phib_1})S,w \rangle|+|\langle R(\alphab)  M^\varphi({\thetab_1,\phib_1})S,w \rangle|\bigg) +\frac{9}{2}\varepsilon^2\|S\|.
    \end{align*}
    From this and $\|S\|\leq 1$ the inequality (ii) follows.
    Finally, inequality (iv) can be shown in a similar way: for every $P \in \PPP$, again using \cref{lem:n2}, it holds that
    \[ 
        \langle   M({\thetab_2,\phib_2})P,w \rangle + \epsilon\cdot\bigg(|\langle M^\theta({\thetab_2,\phib_2})P,w \rangle|+|\langle  M^\varphi({\thetab_2,\phib_2})P,w \rangle|\bigg)+ 2\epsilon^2 \geq \langle  M(\theta_2, \phi_2)P,w \rangle. \qedhere
    \]
\end{proof}

\section{The local theorem} \label{sec:local_thm}

As explained in \cref{sec:outline}, the global theorem cannot possibly handle all regions of the initial search space $I = [0,2\pi) \times [0, \pi] \times [0,2\pi) \times [0, \pi] \times [-\pi/2, \pi/2)$. For some $\theta \in [0, 2 \pi)$ and $\phi \in [0, \pi]$ we would like to examine whether there exists a solution $(\theta_1,\varphi_1,\theta_2,\varphi_2,\alpha)$ to Rupert's problem for $\PPP$ with $\theta_1,\theta_2 \approx \theta$, $\phi_1,\phi_2 \approx \phi$ and $\alpha \approx 0$, since this is one of such cases. More generally, the \emph{local theorem} will be useful if two the projections $\PP$ and $\QQ$ of $\PPP$ look very similar.

\subsection{Definitions and lemmas}
Given $\theta, \phi$, we shall write $X = X(\theta, \phi)$ for the direction of a projection matrix $M = M(\theta, \phi)$. We start with a simple lemma relating $X$ and $M$:

\begin{lem} \label{lem:pythagoras}
    For any $P \in \mathbb{R}^3$ one has
    $\big\|M(\theta, \phi) P\big\|^2=\|P\|^2-\langle X({\theta,\varphi}),P\rangle^2$.
\end{lem}

\begin{proof}
    Define the matrix 
    \[
        A({\theta,\varphi}) \coloneqq\begin{pmatrix}
        M(\theta, \phi)\\
        X({\theta,\varphi})^t
        \end{pmatrix} \in \R^{3 \times 3}.
    \]
    Note that $A({\theta,\varphi}) \cdot A({\theta,\varphi})^t = \id$, thus the matrix is orthogonal and hence it is length-preserving. We therefore obtain
    \[
        \|P\|^2=\|A({\theta,\varphi})P\|^2=\left\|\begin{pmatrix}
        M(\theta, \phi)P\\
        \langle X({\theta, \varphi}),P\rangle
        \end{pmatrix}\right\|^2=\big\|M(\theta, \phi) P\big\|^2+\langle X({\theta, \varphi}),P\rangle^2. \qedhere
    \]
\end{proof}

\begin{dfn}
    Given $v_1, \dots, v_n \in \R^n$ write $\mathrm{span}^+(v_1,\dots,v_n)$ for the set (simplicial cone) in $\R^n$ defined by 
    \[
        \mathrm{span}^+(v_1,\dots,v_n) = \Big\{ w \in \R^n \colon \exists \lambda_1,\dots,\lambda_n > 0 \text{ s.t. } w = \sum_{i=1}^n \lambda_i v_i \Big\},
    \]
    which is the natural restriction of $\mathrm{span}(v_1,\dots,v_n)$ to positive weights.
\end{dfn}

The following lemma lies at the core of the main argument in the proof of the local \cref{thm:local}:
\begin{lem} \label{lem:langles}
Let $V_1,V_2,V_3,Y,Z \in \R^3$ with $\|Y \|=\| Z \|$ and $Y,Z \in \mathrm{span}^+(V_1,V_2,V_3)$. Then at least one of the following inequalities does not hold:
\begin{align*}
    \langle V_1, Y \rangle > \langle V_1, Z \rangle,\\
    \langle V_2, Y \rangle > \langle V_2, Z \rangle,\\
    \langle V_3, Y \rangle > \langle V_3, Z \rangle.
\end{align*}
\end{lem}

\begin{proof}
Let $V = (V_1|V_2|V_3) \in \R^{3 \times 3}$ and assume by contradiction that all inequalities hold. This translates into
\begin{align}\label{ineq}
V^T Y \gg V^T Z,
\end{align}
where we use $x \gg y$ for vectors $x,y\in\R^3$ if $x_1>y_1,x_2>y_2$ and $x_3>y_3.$ Moreover, let $\lambda, \nu \in \R^3_{>0}$ such that $Y = \lambda_1 V_1+\lambda_2V_2+\lambda_3V_3$ and
$Z = \nu_1 V_1+\nu_2V_2+\nu_3V_3$. In other words, we have 
\[
Y =  V \lambda  \quad \text{ and }  \quad Z =  V\nu .
\]
The assumption that $\|Y\| = \|Z\|$ translates therefore into
\[
\lambda^T V^T V \lambda = \nu^T V^T V \nu,
\]
and the inequalities $(\ref{ineq})$ rewrite as
\[
V^T V \lambda \gg V^T V \nu.
\]
Using this and the positivity of $\lambda,\nu \in \R^3$, we find a contradiction in 
\[
    \lambda^T V^T V \lambda > \lambda^T V^T V \nu = (\nu^T V^T V \lambda)^T > (\nu^TV^TV\nu)^T = \lambda^T V^T V \lambda.  \qedhere
\]
\end{proof}

Moreover we will use the following elementary bounds for differences of scalar products:
\begin{lem} \label{lem:scalarprodbars}
    For $A,\overline{A},B,\overline{B}\in \R^{m\times n}$ and $P_1,P_2\in \R^n$ it holds that
    \[
        |\langle AP_1,BP_2\rangle-\langle \overline{A}P_1,\overline{B}P_2\rangle|\leq \|P_1\|\cdot \|P_2\|\cdot \Big( \|A-\overline{A}\|\cdot \|\overline{B}\| +  \|\overline{A}\|\cdot \|B-\overline{B}\|+\|A-\overline{A}\|\cdot \|B-\overline{B}\|\Big).
    \]
\end{lem}

\begin{proof}
    First write
    \begin{align*}
        \langle AP_1,BP_2\rangle&=\langle (A-\overline{A})P_1+\overline{A}P_1,(B-\overline{B})P_2+\overline{B}P_2\rangle\\
        &=\langle (A-\overline{A})P_1,\overline{B}P_2\rangle+\langle \overline{A}P_1,(B-\overline{B})P_2\rangle+\langle (A-\overline{A})P_1,(B-\overline{B})P_2\rangle+\langle \overline{A}P_1,\overline{B}P_2\rangle.
    \end{align*}
    Then the inequality follows from the triangle inequality, the Cauchy-Schwarz inequality and the submultiplicativity of $\|.\|$:
    \begin{align*}
    |\langle AP_1,BP_2\rangle-\langle \overline{A}P_1,\overline{B}P_2\rangle|    &\leq\Big|\langle (A-\overline{A})P_1,\overline{B}P_2\rangle\Big|+\Big|\langle \overline{A}P_1,(B-\overline{B})P_2\rangle\Big|+\Big|\langle (A-\overline{A})P_1,(B-\overline{B})P_2\rangle\Big|\\
    & \hspace{-1cm} \leq\|(A-\overline{A})P_1\|\cdot \|\overline{B}P_2\|+\|\overline{A}P_1\|\cdot \|(B-\overline{B})P_2\|+\|(A-\overline{A})P_1\|\cdot \|(B-\overline{B})P_2\|\\
    & \hspace{-1cm} \leq \|P_1\|\cdot \|P_2\|\cdot \Big( \|A-\overline{A}\|\cdot \|\overline{B}\| +  \|\overline{A}\|\cdot \|B-\overline{B}\|+\|A-\overline{A}\|\cdot \|B-\overline{B}\|\Big). \quad \qedhere
    \end{align*}
\end{proof}

\begin{lem} \label{lem:absscalar}
    For $A,B\in \R^{m\times n}$ and $P_1,P_2\in \R^n$ one has
    $$|\langle AP_1,AP_2\rangle-\langle BP_1,BP_2\rangle|\leq \|P_1\|\cdot \|P_2\|\cdot \|A-B\|\cdot \bigg(\|A\|+\|B\| + \|A-B\|\bigg).$$
\end{lem}

\begin{proof}
    Exactly the same proof as the one for \cref{lem:scalarprodbars} yields
    \begin{align*}
    |\langle AP_1,AP_2\rangle-\langle BP_1,BP_2\rangle|  \leq  \|P_1\|\cdot \|P_2\|\cdot \|A-B\|\cdot \bigg( 2\|B\| + \|A-B\|\bigg)
    \end{align*}
    Exchanging $A$ and $B$, however, also gives the same inequality with $2\|A\|$ instead of $2\|B\|$. Taking the arithmetic mean of these two upper bounds produces the desired symmetric inequality. 
\end{proof}

\subsubsection{Spanning points} \label{sec:spanning}

In order to motivate the definition of spanning points we first prove a simple-to-check algebraic condition for a point to be strictly inside a triangle $ABC$ in $\R^2$. For simplicity, we will only consider the case where the point of interest is the origin. 
Hence, we wish for a criterion on $A,B,C \in \R^2$ such that $(0,0) \in \R^2$ is strictly inside their convex hull. Note that for a non-degenerate triangle this is equivalent to the existence of $a,b,c > 0$ such that
\[
    aA+bB+cC=
    \begin{pmatrix}
        0\\0
    \end{pmatrix}.
\]

\begin{lem} \label{lem:origintriangle}
    Let $A,B,C\in \mathbb{R}^2$ be such that $
    \langle R({\pi/2}) A,B\rangle,
    \langle R({\pi/2}) B,C\rangle,
    \langle R({\pi/2}) C,A\rangle >0$. Then the origin lies strictly in the triangle $ABC$.
\end{lem}

\begin{proof}
    Write $A=(A_1,A_2)^t$, $B=(B_1,B_2)^t$, $C=(C_1,C_2)^t$ and consider the two vectors $v_1=(A_1,B_1,C_1)^t \in \R^3$ and $v_2=(A_2,B_2,C_2)^t \in \R^3$. Now set $(a,b,c)^t \coloneqq v_1\times v_2 \in \R^3$ to be the cross product of $v_1$ and $v_2$. By the orthogonality property of the cross product we obtain
    \[
        aA+bB+cC=\big(A| B| C\big)\cdot (v_1\times v_2)=\begin{pmatrix}
            v_1^t\\ v_2^t
        \end{pmatrix} (v_1\times v_2)=\begin{pmatrix}
            0\\0
        \end{pmatrix}.
    \]
    Additionally observe that
    \[
        \begin{pmatrix}
            a\\b\\c
        \end{pmatrix}= v_1\times v_2 = \begin{pmatrix}
            A_1\\B_1\\C_1
        \end{pmatrix} \times \begin{pmatrix}
            A_2\\B_2\\C_2
        \end{pmatrix}=\begin{pmatrix}
            B_1C_2-C_1B_2\\C_1A_2-A_1C_2\\A_1B_2-B_1A_2
        \end{pmatrix} =\begin{pmatrix}
            \langle R(\pi/2) B,C\rangle\\
            \langle R(\pi/2) C,A\rangle\\
            \langle R(\pi/2) A,B\rangle
        \end{pmatrix},
    \]
    thus $a,b,c >0$.
\end{proof}
This lemma motivates the following definition:

\begin{dfn} \label{def:eps-spanning}
    Let $\theta, \phi \in \R$, $\epsilon>0$ and set $M \coloneqq M(\theta, \phi)$. Three points $P_1, P_2, P_3 \in \R^3$ with $\|P_1\|,\|P_2\|,\|P_3\| \leq 1$ are called \emph{$\epsilon$-spanning for $(\theta, \phi)$} if it holds that:
\begin{align*}
    \langle R(\pi/2) M P_1,M P_{2}\rangle > 2 \epsilon(\sqrt{2} + \epsilon),\\
    \langle R(\pi/2) M P_2,M P_{3}\rangle > 2 \epsilon(\sqrt{2} + \epsilon),\\
    \langle R(\pi/2) M P_3,M P_{1}\rangle > 2 \epsilon(\sqrt{2} + \epsilon).
\end{align*}
\end{dfn}

Then the following lemma translates the $\epsilon$-spanning condition into a condition on $X(\theta, \phi)$:

\begin{lem} \label{lem:eps-spanning}
    Let $P_1, P_2, P_3 \in \R^3$ with $\|P_1\|,\|P_2\|,\|P_3\| \leq 1$ be $\epsilon$-spanning for $(\thetab, \phib)$ and let $\theta, \phi \in \R$ such that $|\theta - \overline{\theta}|, |\phi - \overline{\phi}| \leq \epsilon$. Assume that $\langle X(\theta, \phi), P_i \rangle > 0$ for $i=1,2,3$. Then 
    \[
        X(\theta, \phi) \in \spanp(P_1, P_2, P_3).
    \]
\end{lem}
\begin{proof}
    It follows from the definition of $\epsilon$-spanning and \cref{lem:sqrt2} that 
    \begin{align} \label{eq:originintraiangle}
        \langle R(\pi/2) M(\theta, \phi) P_1,M(\theta, \phi) P_{2}\rangle > 0, \nonumber \\
        \langle R(\pi/2) M(\theta, \phi) P_2,M(\theta, \phi) P_{3}\rangle > 0, \\
        \langle R(\pi/2) M(\theta, \phi) P_3,M(\theta, \phi) P_{1}\rangle > 0, \nonumber
    \end{align}
    using \cref{lem:scalarprodbars} on $A=R(\pi/2) M(\theta, \phi), B =  M(\theta, \phi)$ analogously for $\overline{A}, \overline{B}$ with the fact that $\|R(\pi/2)\|=1$ and $\|P_1\|,\|P_2\|,\|P_3\| \leq 1$. Thus, \cref{lem:origintriangle} implies that the origin $O$ is strictly inside the triangle $\Delta(M(\theta, \phi) P_1, M(\theta, \phi) P_2, M(\theta, \phi) P_3)$.
    It follows that there exist $a,b,c >0$ such~that 
    \begin{align} \label{eq:Ybarkernel}
        aM(\theta,\phi)P_1 + bM(\theta,\phi)P_2 + cM(\theta,\phi)P_3 = O.
    \end{align}
    Consider the point $S = aP_1 + bP_2 + cP_3\in \R^3$. From (\ref{eq:Ybarkernel}) it follows that $S \in \ker M(\theta,\phi)$. However, since $X \coloneqq X(\theta, \phi) \in \ker M(\theta,\phi)$ as well, we get that $S = \lambda X(\theta, \phi)$ for some $\lambda \in \R$. It follows that
    \[
        \lambda = \langle X, \lambda X \rangle = \langle X, S \rangle = a \langle X, P_1\rangle + b \langle X, P_2\rangle + c \langle X, P_3\rangle > 0,
    \]
    and, consequently, that $X(\theta, \phi) = \frac{1}{\lambda}S \in \spanp(P_1, P_2, P_3)$. 
\end{proof}

\subsubsection{Locally maximally distant points} \label{sec:lmds}

In the situation of \cref{lem:eps-spanning}, the combination of \cref{lem:pythagoras} and \cref{lem:langles} allows us to show that the distances of the projections $\epsilon$-spanning points cannot all simultaneously increase. 
Intuitively, this is almost enough to exclude Rupert's property for a region $[\thetab_1 \pm \epsilon,\phib_1 \pm \epsilon, \thetab_1 \pm \epsilon, \phib_1 \pm \epsilon, 0]$, as if there was a solution, the ``smaller'' projection should have all three distances smaller. This is, however, not true in general, as shown in Figure~\ref{fig:LMD2} -- it holds that $\|OA\| > \|O Q_2\|$, even though $A \in \PP^\circ$. 
Therefore, we introduce the concept of \emph{locally maximally distant points} -- vertices of a polygon which are locally the farthest away from the origin. 
The precise \cref{def:LMD} is more technical, since we need to take care of the $\epsilon$-region.
Moreover, in \cref{lem:LMD} we will show that a vertex $P \in \PP$ is locally maximally distant if all the angles $\angle(O,P,P_j)$ are acute (for all $P_j \in \PP \setminus P$) -- evidently this is exactly the condition that is not fulfilled for $Q_2$ in Figure~\ref{fig:LMD2}.

\begin{dfn} 
    For $\delta > 0$ and $\overline{Q} \in \R^2$ we define $\Circ_{\delta}(\overline{Q})$ to be the open disk with center $\overline{Q}$ and radius~$\delta$. 
\end{dfn}

\begin{lem} \label{lem:inCirc}
    Let $P, Q \in \R^3$ with $\|P\|, \|Q\| \leq 1$. Let $\epsilon>0$ and $\thetab_1,\phib_1,\thetab_2,\phib_2,\alphab \in \R$, then set
    \[
        T \coloneqq \left(R(\alphab) M(\thetab_1, \phib_1) P + M(\thetab_2, \phib_2) Q\right)/2 \in \R^2,
    \]
    and $\delta \geq \|T - M(\thetab_2, \phib_2) Q\|$. Finally, let $\theta_1, \phi_1, \theta_2, \phi_2, \alpha \in \R$ with $|\thetab_1-\theta_1|, |\phib_1 - \phi_1|, |\thetab_2-\theta_2|, |\phib_2-\phi_2|, |\alphab - \alpha| \leq \epsilon$. Then $R(\alpha)M(\theta_1, \phi_1) P, M(\theta_2, \phi_2) Q \in \Circ_{\delta + \sqrt{5} \epsilon}(T)$.
\end{lem}
\begin{proof}
    By the triangle inequality and \cref{lem:sqrt5} we have
    \begin{align*}
        \| R(\alpha) M(\theta_1, \phi_1) P - T \| & \leq \| R(\alpha) M(\theta_1, \phi_1) P - R(\alphab) M(\thetab_1, \phib_1) P \| + \| R(\alphab) M(\thetab_1, \phib_1) P - T \|\\
        & \leq \| R(\alpha) M(\theta_1, \phi_1) - R(\alphab) M(\thetab_1, \phib_1) \| \cdot \|P\| + \delta < \sqrt{5} \epsilon + \delta.
    \end{align*}
    Similarly, using \cref{lem:sqrt2} we also find that 
    \[
        \| M(\theta_2, \phi_2) Q - T \| < \sqrt{2} \epsilon + \delta < \sqrt{5} \epsilon + \delta. \qedhere
    \]
\end{proof}

\begin{dfn} \label{def:LMD}
    Let $\PP \subset \R^2$ be a convex polygon and $Q \in \PP$ one of its vertices. Assume that for some $\overline{Q} \in \R^2$ it holds that $Q \in \Circ_{\delta}(\overline{Q})$, i.e. $\|Q - \overline{Q}\| < \delta$. Define $\Sect_\delta(\overline{Q}) \coloneqq \Circ_{\delta}(\overline{Q}) \cap \PP^\circ$ as the intersection between $\Circ_{\delta}(\overline{Q})$ and the interior of the convex hull of $\PP$.
    
    Moreover, $Q \in \PP$ is called \emph{$\delta$-locally maximally distant with respect to $\overline{Q}$ ($\delta$-LMD$(\overline{Q})$)} if for all $A \in \Sect_\delta(\overline{Q})$ it holds that $\|Q\| > \|A\|$.
\end{dfn}

\begin{figure}[h]
    \centering
    \includegraphics[width=1\linewidth]{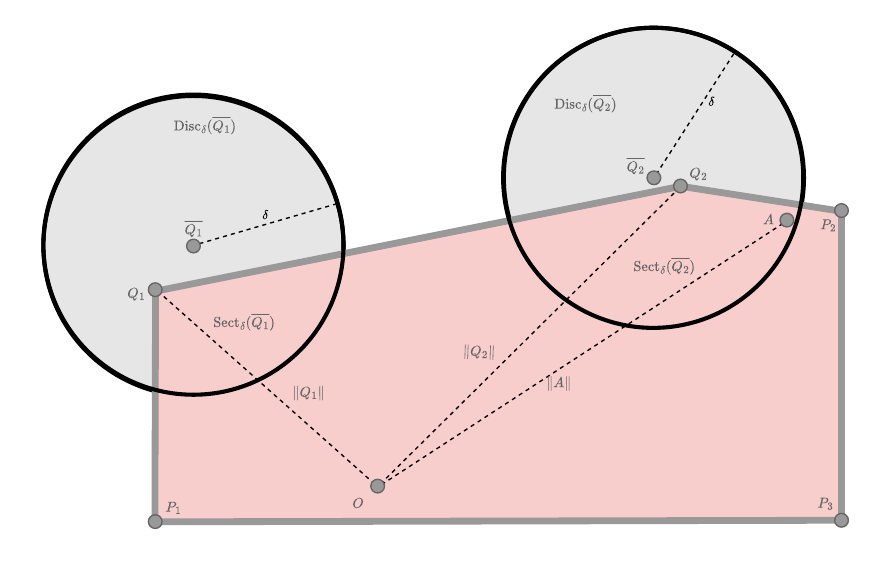}
    \caption{A polygon $\PP = \{P_1, Q_1, Q_2, P_2, P_3\}  \subset \R^2$. Here $O$ is the origin. Note that $Q_1$ is $\delta$-LMD but $Q_2$ is not, since $\|A\| > \|Q_2\|$ and $A \in \Sect_\delta(\overline{Q_2})$. As it is apparent from the proof of \cref{lem:LMD} this is because $\angle(O Q_2 P_2)$ is obtuse.}
    \label{fig:LMD2}
\end{figure}

Figure~\ref{fig:LMD2} is a visualization of this definition, in particular of $\Sect_\delta(Q)$. The figure indicates that the $\delta$-LMD condition is not trivial to check. For that purpose we introduce the following lemma:

\begin{lem} \label{lem:LMD}
    Let $\PP$ be a convex polygon and $Q \in \PP$ be one of its vertices. Let $\overline{Q} \in \R^2$ with $\|Q - \overline{Q}\| < \delta$ for some $\delta>0$. Assume that for some $r > 0 $ such that $\|Q\| > r$ it holds that 
    \[
        \frac{\langle Q, Q - P_j \rangle}{\|Q\|\|Q - P_j\|} \geq \delta/r,
    \]
    for all other vertices $P_j \in \PP \setminus Q$. Then $Q \in \PP$ is $\delta$-locally maximally distant with respect to $\overline{Q}$.
\end{lem}

\begin{proof}
    Let $A \in \Sect_\delta(\overline{Q}) = \Circ_\delta(\overline{Q}) \cap \PP^\circ$. We need to show that $\|A\| < \|Q\|$. By the cosine theorem in the triangle $\Delta OQA$ it holds that 
    \[
        \| A \|^2 = \|Q\|^2 + \|Q - A\|^2 - 2\|Q\|\|Q - A\|\cos(\angle(O,Q,A)).
    \]
    Or, equivalently, that 
    \[
        \|A\|^2 - \|Q\|^2 = \|Q - A\| \cdot (\|Q - A\| - 2\|Q\|\cos(\angle(O,Q,A))).
    \]
    Now observe that $\angle(O,Q,A) < \angle(O,Q,P_j)$ for some vertex $P_j \in \PP \setminus Q$ because $A \in \PP^\circ$. Moreover, recall that 
    \[
    \cos(\angle(O,Q,P_j)) = \frac{\langle Q, Q - P_j \rangle}{\|Q\|\|Q - P_j\|} \geq \delta/r,
    \]
    thus $\cos(\angle(O,Q,A)) \geq \delta/r$. We also have that $\|Q\| \geq r$ by assumption and $\|Q - A\| < 2\delta$ because both $Q, A \in \Circ_\delta(\overline{Q})$. We conclude that 
    \[
        \|Q - A\| - 2\|Q\|\cos(\angle(O,Q,A)) < 2\delta - 2r\cdot\delta/r = 0,
    \]
    therefore $\|A\|^2 - \|Q\|^2 < 0$ and $\|A\| < \|Q\|$
\end{proof}

We will also use the following fact:
\begin{lem} \label{lem:coss}
    Let $\epsilon>0$ and $\theta,\thetab, \phi, \phib \in \R$ with $|\theta - \overline{\theta}|, |\phi - \overline{\phi}| \leq \epsilon$. Define $M = M(\theta, \phi)$ and $\overline{M} = M(\thetab, \phib)$ and let $P, Q \in \R^3$ with $\|P\|, \|Q\| \leq 1$. Assume\footnote{In a first version of this manuscript we were lacking this assumption and we thank \href{https://dwrensha.ws/}{David Renshaw} for pointing~it~out.} that 
    \[
    \frac{\langle \overline{M} P,\overline{M} (P-Q)\rangle - 2 \epsilon \|P-Q\| \cdot  (\sqrt{2}+\varepsilon)}{ \big(\|\overline{M} P\|+\sqrt{2} \varepsilon \big) \cdot \big(\|\overline{M}(P-Q)\|+2 \sqrt{2} \varepsilon\big)} > 0, 
    \]
    then:
    \[
        \frac{\langle {M} P,{M} (P-Q)\rangle}{\|{M} P\| \cdot \|{M}(P-Q)\|} \geq 
        \frac{\langle \overline{M} P,\overline{M} (P-Q)\rangle - 2 \epsilon \|P-Q\| \cdot  (\sqrt{2}+\varepsilon)}{ \big(\|\overline{M} P\|+\sqrt{2} \varepsilon \big) \cdot \big(\|\overline{M}(P-Q)\|+2 \sqrt{2} \varepsilon\big)}. 
    \]
\end{lem}
\begin{proof}
We will show that the numerator of the left-hand side is bigger than the one of the right-hand side, and that the factors in the denominator are smaller. We also show that the left-hand side is well-defined as $MP, M(P-Q) \neq 0$.

We apply \cref{lem:absscalar} to $A = M, B = \overline{M}$ and $P_1 = P, P_2 = P-Q$:
\begin{align*}
|\langle {M} P,{M} (P-Q) \rangle - \langle \overline{M} P,\overline{M} (P-Q)\rangle | & \leq \|P\| \cdot \|P - Q\| \cdot \| M - \overline{M} \| \cdot (\|M\| + \|\overline{M}\| + \| M - \overline{M} \|) \\
& \leq 1 \cdot \| P - Q \| \cdot \sqrt{2}\epsilon \cdot (1 + 1 + \sqrt{2} \epsilon),
\end{align*}
and observe the claim for the numerator using the triangle inequality. In particular, $\langle {M} P,{M} (P-Q)\rangle > 0$, thus $MP$ and $M(P-Q)$ are non-zero, and the denominator does not~vanish.

For the first factor in the denominator we simply have 
\[
\| \overline{M} P \| \geq \| M P \| - \| MP - \overline{M}P \| \geq \| M P \| - \| M - \overline{M} \| \cdot \|P\| \geq \|MP\| - \sqrt{2} \epsilon,
\]
with the triangle inequality and \cref{lem:sqrt2}. For the other factor we obtain similarly 
\[
\| \overline{M} (P-Q) \| \geq \| M (P-Q) \| - \| M - \overline{M} \| \cdot \|P-Q\| \geq \|M(P-Q)\| - 2 \sqrt{2} \epsilon,
\]
using that $\|P-Q\| \leq 2$.
\end{proof}

\subsubsection{Congruent points}
\begin{dfn}
Two triples of points $P_1, P_2, P_3 \in \R^3$ and $Q_1, Q_2, Q_3 \in \R^3$ are called \emph{congruent} if there exists an orthonormal matrix $L$ such that $P_i = LQ_i$ for $i=1,2,3$.
\end{dfn}

\begin{lem} \label{lem:congruent}
    Let $P_1,P_2,P_3, Q_1,Q_2,Q_3 \in \R^3$. Define the $3 \times 3$ matrices $P \coloneqq (P_1|P_2|P_3)$ and $Q \coloneqq (Q_1|Q_2|Q_3)$ and assume that $Q$ is invertible. 
    Then $P_1, P_2, P_3$ and $Q_1, Q_2, Q_3$ are congruent if and only if $P^t P = Q^t Q$. 
\end{lem}
Note that $P^t P = Q^t Q$ is equivalent to saying that $\langle P_i,P_j\rangle = \langle Q_i,Q_j\rangle$ for all $1 \leq i,j \leq 3$. Moreover, the condition on invertibility of $Q$ can be dropped, however then the proof becomes somewhat less straightforward.

\begin{proof}
    If $P_1, P_2, P_3$ and $Q_1, Q_2, Q_3$ are congruent then $\langle P_i,P_j\rangle = \langle LQ_i,LQ_j\rangle = \langle Q_i,Q_j\rangle$, thus $P^t P = Q^t Q$.
    For the other direction, we claim that $L \coloneqq PQ^{-1}$ is orthonormal and satisfies that $P_i = LQ_i$ for all $i=1,2,3$.  
    Indeed,  $L^t L = (PQ^{-1})^t (PQ^{-1}) = (Q^t)^{-1} P^t P Q^{-1} = \mathrm{Id}$ and it holds that $LQ=PQ^{-1}Q=P$, thus $LQ_i = P_i$.
\end{proof}

\subsection{Statement and proof of the local theorem}

\begin{thm}[Local Theorem] \label{thm:local}
    Let $\PPP$ be a polyhedron with radius $\rho=1$ and $P_1, P_2, P_3, Q_1, Q_2, Q_3 \in \PPP$ be not necessarily distinct. Assume that $P_1, P_2, P_3$ and $Q_1, Q_2, Q_3$ are congruent.

    Let $\epsilon>0$ and $\thetab_1,\phib_1,\thetab_2,\phib_2,\alphab \in \R$, then set $\Xib \coloneqq X(\thetab_1,\phib_1), \Xiib \coloneqq X(\thetab_2,\phib_2)$ as well as $\Mib \coloneqq M(\thetab_1,\phib_1), \Miib \coloneqq M(\thetab_2,\phib_2)$.
    Assume that there exist $\sigma_P, \sigma_Q \in \{0,1\}$ such that 
    \[
        (-1)^{\sigma_P} \langle \Xib,P_i\rangle>\sqrt{2}\varepsilon \quad \text{and} \quad
        (-1)^{\sigma_Q} \langle \Xiib , Q_i\rangle>\sqrt{2}\varepsilon, \tag{A$_\epsilon$}
    \]
    for all $i=1,2,3$.
    Moreover, assume that $P_1,P_2,P_3$ are $\epsilon$-spanning for $(\thetab_1,\phib_1)$ and that $Q_1,Q_2,Q_3$ are $\epsilon$-spanning for $(\thetab_2,\phib_2)$.
    Finally, assume that for all $i = 1,2,3$ and any $Q_j \in \PPP \setminus Q_i$ it holds~that 
    \[
        \frac{\langle \Miib Q_i,\Miib (Q_i-Q_j)\rangle - 2 \epsilon \|Q_i-Q_j\| \cdot  (\sqrt{2}+\varepsilon)}{ \big(\|\Miib Q_i\|+\sqrt{2} \varepsilon \big) \cdot \big(\|\Miib(Q_i-Q_j)\|+2 \sqrt{2} \varepsilon\big)} > \frac{\sqrt{5} \epsilon + \delta}{r}, \tag{B$_\epsilon$}
    \]
    for some $r >0$ such that $\min_{i=1,2,3}\| \Miib Q_i \| > r + \sqrt{2} \epsilon$ and for some $\delta \in \R$ with 
    \[
        \delta \geq \max_{i=1,2,3}\left\|R(\alphab) \Mib P_i - \Miib Q_i\right\|/2.
    \]
    Then there exists no solution to Rupert's problem $R(\alpha) M(\theta_1,\phi_1)\PPP \subset  M(\theta_2,\phi_2)\PPP^\circ$ with 
    \[
        (\theta_1, \phi_1, \theta_2, \phi_2, \alpha) \in [\thetab_1\pm\epsilon,\phib_1\pm\epsilon,\thetab_2\pm\epsilon,\phib_2\pm\epsilon,\alphab\pm\epsilon] \coloneqq U \subseteq \R^5.
    \]
\end{thm}

Before we prove this statement, we wish to provide an example. 
\begin{example}
Consider the Octahedron $\OOO = \{ (\pm 1, 0,0), (0, \pm1,0),(0,0,\pm1) \} \subseteq \R^3$ and the projection direction $(\thetab, \phib) = (\pi/4,\tan^{-1}(\sqrt{2}))$. The polygon $M(\thetab, \phib) \OOO$ is depicted in Figure~\ref{fig:octahedron}. We enumerate the vertices $O_1,\dots,O_6$ of $\OOO$ in such a way that $O_1,O_2,O_3$ have nonnegative coordinates. The goal of this example is to use \cref{thm:local} to show that there exists an $\epsilon>0$ such that there is no solution $\Psi$ to Rupert's problem for $\OOO$ with 
\[
\Psi = (\theta_1, \phi_1, \theta_2, \phi_2, \alpha) \in U = [\phib \pm \epsilon, \thetab \pm \epsilon, \phib \pm \epsilon, \thetab \pm \epsilon, \pm \epsilon].
\]

\begin{figure}
    \centering
    \includegraphics[width=\linewidth]{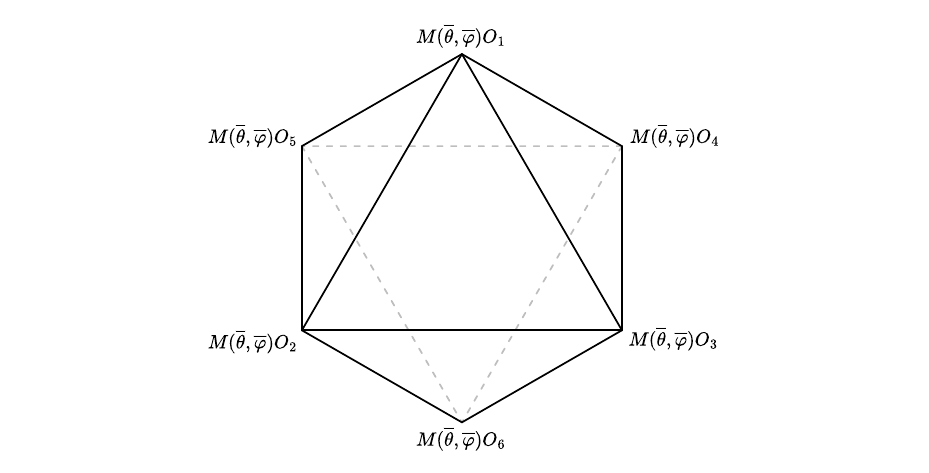}
    \caption{The projection of the octahedron $\OOO$ in direction $(\thetab, \phib) = (\pi/4,\tan^{-1}(\sqrt{2})$.}
    \label{fig:octahedron}
\end{figure}

In order to apply \cref{thm:local} we choose $P_i=Q_i = O_i$ for $i=1,2,3$ as well as $\thetab_1 = \thetab_2 = \thetab = \pi/4$, $\phib_1 = \phib_2 = \phib = \tan^{-1}(\sqrt{2})$ and $\alphab=0$. This means that
\[
\Mib = \Miib = \begin{pmatrix}
    -\sqrt{2}/2 & \sqrt{2}/2 & 0 \\
    -\sqrt{6}/6 & -\sqrt{6}/6 & \sqrt{6}/3
\end{pmatrix} \quad \text{ and } \quad \Xib = \Xiib = \begin{pmatrix} 1 & 1 & 1 \end{pmatrix}^{t}/\sqrt{3}.
\]
Since $\langle \Xib, O_i \rangle = 1/\sqrt{3}$ for $i=1,2,3$ it follows that $\sigma_P = \sigma_Q=0$ and $(A_\epsilon)$ is fulfilled for epsilon small enough. In Figure~\ref{fig:octahedron} this simply means that the points $\Mib O_i$ for $i=1,2,3$ are ``in front'' of the projection. 

Next we need to check that $O_1,O_2,O_3$ are $\epsilon$-spanning for $(\thetab, \phib)$. Heuristically speaking, this simply means that the origin is inside the triangle $\Delta(\Mib O_1,\Mib O_2,\Mib O_3)$ even after exchanging $\thetab,\phib$ with any possible $(\theta, \phi) \in [\thetab \pm \epsilon,\phib \pm \epsilon]$. To verify this rigorously, we compute
\[
\langle R(\pi/2) \Mib O_1, \Mib O_2\rangle = 1/\sqrt{3},
\]
and similarly for the other two conditions in \cref{def:eps-spanning}. Thus, the $\epsilon$-spanning condition is satisfied as long as $1/\sqrt{3} > 2 \epsilon(\sqrt{2} + \epsilon)$.

For the last condition (B$_\epsilon$) if we fix without loss of generality $Q_i = O_1$ the left-hand side becomes:
\begin{align*}
  \frac{\langle \Miib O_1,\Miib (O_1-O_j)\rangle - 2 \epsilon \|O_1-O_j\| \cdot  (\sqrt{2}+\varepsilon)}{ \big(\|\Miib O_1\|+\sqrt{2} \varepsilon \big) \cdot \big(\|\Miib(O_1-O_j)\|+2 \sqrt{2} \varepsilon\big)}  = 
  \begin{cases}
  \frac{1-2 \epsilon  \sqrt{2}\, (\sqrt{2}+\epsilon )}{(\sqrt{6}/{3}+\epsilon  \sqrt{2}) (\sqrt{2}+2 \epsilon  \sqrt{2})} \quad \text{ for } j=2,3,\\
  \frac{{1}/{3}-2 \epsilon  \sqrt{2}\, (\sqrt{2}+\epsilon )}{({\sqrt{6}}/{3}+\epsilon  \sqrt{2}) \left({\sqrt{6}}/{3}+2 \epsilon  \sqrt{2}\right)} \quad \text{ for } j=4,5,\\
  \frac{{4}/{3}-4 \epsilon  (\sqrt{2}+\epsilon )}{({\sqrt{6}}/{3}+\epsilon  \sqrt{2}) ({2 \sqrt{6}}/{3}+2 \epsilon  \sqrt{2})} \quad \text{ for } j=6.
  \end{cases}
\end{align*}
Note that for $\epsilon$ small these expressions are approximately $\sqrt{3}/2, 1/2$ and $1$ respectively. For the right-hand side note that $\delta=0$ and if we choose $r < \|\Mib O_i\| = \sqrt{6}/3$ there will exist $\epsilon>0$ which satisfies $\|\Mib O_i\| > r + \sqrt{2}\epsilon$. Moreover, since $\sqrt{5}\epsilon/r$ becomes arbitrarily small, (B$_\epsilon$) holds for small enough $\epsilon$. Following \cref{sec:lmds} this computation simply shows that $\Mib O_i$ are locally maximally distant points (also after exchanging $\thetab,\phib$ with any possible $\theta, \phi \in [\thetab \pm \epsilon, \phib \pm \epsilon]$), because the angles $\angle (O,O_1,O_j)$ for $j=2,3,4,5,6$ in Figure~\ref{fig:octahedron} are approximately $60^\circ, 30^\circ$, or $0^\circ$, in particular acute.

Altogether we have have proven that there is no ``local'' solution to Rupert's problem for the Octahedron from the direction $(\thetab, \phib) = (\pi/4,\tan^{-1}(\sqrt{2})$. We can also effectively choose, for example, $r=0.7$ and $\epsilon=0.05$ to satisfy all appearing inequalities.

We also mention that if $(\thetab_1, \phib_1) = (\pi/4, \tan^{-1}(\sqrt{2}))$ but $(\thetab_2, \phib_2) = (\pi/4, \pi - \tan^{-1}(\sqrt{2}))$ then the two projections look still like in Figure~\ref{fig:octahedron}, however then we would need to take $P_i = O_i$ for $i=1,2,3$ but $(Q_1,Q_2,Q_3) = (O_1, O_5, O_4)$ and verify that these points are congruent. Also, in this case, $\sigma_Q = 1$.

Finally, note that the proof of \cref{thm:local} crucially uses \cref{lem:langles} in combination with \cref{lem:pythagoras}, which in the context of this example simply state that the three distances $\|MO_1\|$, $\|MO_2\|$, $\|MO_3\|$ cannot all (locally) simultaneously increase. An even simpler way to see this exists, since $O_1, O_2, O_3$ are orthogonal to each other and of equal length -- in this case it holds that 
\[
    \|MO_1\|^2 + \|MO_2\|^2 + \|MO_3\|^2 = 2.
\]
\end{example}

\bigskip

\begin{proof}[Proof of \cref{thm:local}]
    Assume by contradiction that there exists a $\Psi = (\theta_1, \phi_1, \theta_2, \phi_2, \alpha) \in U$ which solves Rupert's problem. By assumption $(P_1,P_2,P_3)$ and $(Q_1,Q_2,Q_3)$ are congruent, thus there exists an orthonormal matrix $L \in \R^{3 \times 3}$ with $P_i  = LQ_i$ for all $i=1,2,3$. Set $Y \coloneqq X(\theta_1, \phi_1)$ and $Z \coloneqq K X(\theta_2, \phi_2)$, where $K \coloneqq (-1)^{\sigma_P + \sigma_Q}L$, so it is easy to see that $\|Y\|=\|Z\|=1$. Moreover, let $\overline{P_i} = (-1)^{\sigma_P} P_i$ and $\overline{Q_i} = (-1)^{\sigma_Q} Q_i$, so that $\overline{P_i} = K\overline{Q_i}$. 
    
    We shall show that 
        \begin{align} \label{eq:YZinspanp_fin}
            Y,Z \in \spanp(\overline{P_1},\overline{P_2},\overline{P_3})
        \end{align}
        and that
        \begin{align} \label{eq:angles_fin} 
            \langle Y, \overline{P_i} \rangle > \langle Z, \overline{P_i} \rangle \text{ for each } i=1,2,3,
        \end{align}
        then \cref{lem:langles} yields the desired contradiction. \\    

    First observe that (A$_\epsilon$) and \cref{lem:XP>0} imply that
    \begin{equation} \label{eq:YPiZPi>0_fin}
        \langle X(\theta_1, \phi_1), \overline{P_i} \rangle, \langle X(\theta_2, \phi_2), \overline{Q_i} \rangle > 0.
    \end{equation}
    Now (\ref{eq:YZinspanp_fin}) for $Y = X(\theta_1, \phi_1)$ follows from \cref{lem:eps-spanning}. By the same argument we also get that 
    \[
        X(\theta_2, \phi_2) \in \spanp(\overline{Q_1},\overline{Q_2},\overline{Q_3}).
    \]
    However, then (\ref{eq:YZinspanp_fin}) also immediately follows for $Z = K X(\theta_2, \phi_2)$ as $\overline{P_i} = K \overline{Q_i}$.\\

    We proceed with proving (\ref{eq:angles_fin}). First note that by assumption $\| \Miib Q_i \| > r + \sqrt{2} \epsilon$, thus  \cref{lem:MP>r} shows that 
    \(
        \| M(\theta_2, \phi_2) Q_i \| > r.
    \)
    Define 
    \[
        T_i \coloneqq \left(R(\alphab) \Mib P_i + \Miib Q_i\right)/2 \in \R^2,
    \]
    for $i=1,2,3$ so that by assumption $\|T_i - \Miib Q_i\| \leq \delta$. Then also set
    \[
        \Circ_i \coloneqq \Circ_{\delta + \sqrt{5} \epsilon}(T_i) = \{ x \in \R^2: \| x - T_i \| < \delta + \sqrt{5} \epsilon \}.
    \] 
    By \cref{lem:inCirc} it holds that $R(\alpha)M(\theta_1, \phi_2) P_i, M(\theta_2, \phi_2) Q_i \in \Circ_i$.
    Moreover, as the right-hand side in (B$_\epsilon$) is positive, \cref{lem:coss} implies that 
    \[
    \frac{\langle M(\theta_2, \phi_2) Q_i,M(\theta_2, \phi_2) (Q_i-Q_j)\rangle}{\|M(\theta_2, \phi_2) Q_i\| \cdot \|M(\theta_2, \phi_2)(Q_i-Q_j)\|} \geq 
    \frac{\langle \Miib Q_i,\Miib (Q_i-Q_j)\rangle - 2 \epsilon \|Q_i-Q_j\| \cdot  (\sqrt{2}+\varepsilon)}{ \big(\|\Miib Q_i\|+\sqrt{2} \varepsilon \big) \cdot \big(\|\Miib(Q_i-Q_j)\|+2 \sqrt{2} \varepsilon\big)},
    \]
    Thus, using assumption (B$_\epsilon$), the left-hand side of this inequality is also larger than $(\delta + \sqrt{5}\epsilon)/r$. \cref{lem:LMD} now implies that $M(\theta_2, \phi_2) Q_i \in M(\theta_2, \phi_2)\PPP$ is $(\delta + \sqrt{5}\epsilon$)-LMD with respect to $T_i$. We obtain that for any $A \in \Circ_i \cap (M(\theta_2, \phi_2) \PPP)^\circ$ it holds that $\|A\| < \|M(\theta_2,\varphi_2)Q_i\|$.

    Recall that by contradiction we assumed that $\Psi$ is a solution to Rupert's problem, thus, in particular, $R(\alpha)M(\theta_1,\varphi_1)P_i \in (M(\theta_2, \phi_2) \PPP)^\circ$ and it follows that $A = R(\alpha)M(\theta_1,\varphi_1)P_i \in \Circ_i \cap (M(\theta_2, \phi_2) \PPP)^\circ$.
    We conclude that $ \| M(\theta_1,\varphi_1)P_i \| < \|M(\theta_2,\varphi_2)Q_i\|$. Now \cref{lem:pythagoras} implies that $ \langle X(\theta_1, \phi_1), P_i \rangle^2 > \langle X(\theta_2, \phi_2), Q_i \rangle^2$ and (\ref{eq:YPiZPi>0_fin}) implies that 
    \[
        \langle X(\theta_1, \phi_1), \overline{P_i} \rangle > \langle X(\theta_2, \phi_2), \overline{Q_i} \rangle.
    \]
    We conclude that (\ref{eq:angles_fin}) holds because
    \[
        \langle Y, \overline{P_i} \rangle = \langle X(\theta_1, \phi_1), \overline{P_i} \rangle>
        \langle X(\theta_2, \phi_2), \overline{Q_i} \rangle
        = 
        \langle K X(\theta_2, \phi_2), K\overline{Q_i} \rangle
        = 
        \langle Z, \overline{P_i} \rangle. \qedhere
    \]
\end{proof}

\section{Rational versions} \label{sec:rational}

We explained in \cref{sec:outline} that \cref{thm:global} and \cref{thm:local} cannot be used directly on a computer with floating point arithmetic. This is because the appearing inequalities in the assumptions are in practice very tight and finite machine precision on 64 bits cannot rigorously guarantee their validity. For this reason in this section we will prove adaptations of the two theorems in which we will approximate all real expressions with rational numbers. We set $\kappa \coloneqq 10^{-10}$ that  will be used throughout the remaining text, and we will write $\widetilde{x}\in \mathbb{Q}^{m\times n}$ for a rational approximation of $x\in \mathbb{R}^{m \times n}$ meaning that $\|x-\widetilde{x}\|\leq \kappa$. Moreover, $\widetilde{\PPP} \subseteq \Q^3$ is a $\kappa$-rational approximation of $\PPP \subseteq \R^3$ if each $\widetilde{P}_i \in \widetilde{\PPP}$ is a $\kappa$-rational approximation of the corresponding ${P}_i \in {\PPP}$.

\subsection{Definitions and lemmas}

\begin{dfn}
    We define the two functions $\ssin, \scos: \R \to \R$ by:
    \begin{align*}
    \ssin(x) & \coloneqq x-\frac{x^3}{3}+\frac{x^5}{5!}\mp \dots +\frac{x^{25}}{25!},\\
    \scos(x) & \coloneqq 1-\frac{x^2}{2}+\frac{x^4}{4!}\mp\dots +\frac{x^{24}}{24!}.
    \end{align*}
    Further, by replacing $\sin,\cos$ with $\ssin,\scos$ we define the functions 
\[
    R_{\Q}(\alpha), R'_{\Q}(\alpha), X_{\Q}(\theta, \phi), M_{\Q}(\theta, \phi), M_{\Q}^{\theta}(\theta,\varphi),M_{\Q}^{\phi}(\theta,\varphi).
\]
\end{dfn}
Note that $\ssin, \scos$ map $\Q$ to $\Q$. Moreover, by the following lemma they are good approximations for $\sin, \cos$ for input values in $[-4,4]$:

\begin{lem} \label{lem:kappa/7}
    For every $x\in [-4,4]$ it holds that
    \[
    |\ssin(x)-\sin(x)| \leq \frac{\kappa}{7} \quad \text{and} \quad |\scos(x)-\cos(x)|\leq \frac{\kappa}{7}.
    \]
\end{lem}
\begin{proof}
    It is well known and not difficult to prove that for any $x \in \mathbb{R}$ one has
    \[
        |\ssin(x)-\sin(x)|\leq \frac{|x|^{27}}{27!} \quad \text{as well as} \quad
    |\scos(x)-\cos(x)|\leq \frac{|x|^{26}}{26!}.
    \]
    As 
    ${4^{26}/26!}\approx 1.1\cdot 10^{-11}$ and ${4^{27}/27!}\approx 1.7\cdot 10^{-12}$ and $\kappa/7\approx 1.4\cdot 10^{-11}$ the claim is proven.
\end{proof}

The following useful lemma is an elementary bound on $\|A\|$ in terms of the entries of $A$:

\begin{lem} \label{lem:A<deltamn}
Let $A = (a_{i,j})_{1 \leq i \leq m,\ 1 \leq j \leq n} \in \R^{m \times n}$ and $\delta >0$. Assume that $|a_{i,j}| \leq \delta$. Then it holds that $\|A\| \leq \delta \sqrt{mn}.$
\end{lem}

\begin{proof}
    For any $v\in \R^n$ we have
    \begin{align*}
        \|Av\|^2 &=\sum_{i=1}^m \left(\sum_{j=1}^na_{i,j}v_j\right)^2 \leq \sum_{i=1}^m\left(\sum_{j=1}^n \delta |v_j|\right)^2 = \delta^2 m\left(\sum_{j=1}^n |v_j|\right)^2 \leq \delta^2 m n \|v\|^2
    \end{align*}
    using the Cauchy-Schwarz inequality. Dividing by $\|v\|$ and taking the square root proves the claim.
\end{proof}
    
\begin{lem}\label{lem:dist<kappa}
    Let $A(x,y)$ be an $m\times n$ matrix with $1 \leq m,n\leq 3$ such that every entry is of the form $a_1(x)\cdot a_2(y)$ where 
    $a_i(z)\in \{0,1,-1,\pm\sin(z),\pm\cos(z)\}.$
    Define $A_{\mathbb{Q}}(x,y)$ by replacing $\sin$ with $\ssin$ and $\cos$ with $\scos$. Then for every $x,y\in[-4,4]$ it holds that
    $\|A(x,y)-A_\mathbb{Q}(x,y)\|\leq\kappa$.
\end{lem}

\begin{proof} 
    For fixed $x,y$ every entry of $A(x,y)-A_{\mathbb{Q}}(x,y)$ is of the form
    $a b - \widetilde{a}\widetilde{b}$ for some $a,b\in[-1,1]$ and $|a-\widetilde{a}|,|b-\widetilde{b}|\leq \kappa/7$ by \cref{lem:kappa/7}. This implies that
    \begin{align*}
      |ab-\widetilde{a}\widetilde{b}|&\leq |a b-a\widetilde{b}|+|a \widetilde{b}-\widetilde{a}\widetilde{b}|
      =|a|\cdot |b-\widetilde{b}|+|\widetilde{b}|\cdot |a-\widetilde{a}| \leq 1\cdot \kappa/7+(1+\kappa/7) \cdot \kappa/7 <\kappa/3.
    \end{align*}
    So we can apply \cref{lem:A<deltamn} and obtain that $\|A(x,y)-A_{\Q}(x,y)\|<\kappa/3\cdot \sqrt{3\cdot 3}=\kappa$.
\end{proof}

Now note that all the functions 
\[
    R_{\Q}(\alpha), R'_{\Q}(\alpha), X_{\Q}(\theta, \phi), M_{\Q}(\theta, \phi), M_{\Q}^{\theta}(\theta,\varphi),M_{\Q}^{\phi}(\theta,\varphi)
\]
are of the form as in \cref{lem:dist<kappa}, hence we get as a corollary:

\begin{corr} \label{corr:kappa1kappa}
    Let $\alpha,\theta,\phi\in [-4,4]$. Then it holds that
    \begin{align*}
        \|R(\alpha)-R_{\Q}(\alpha)\|, \|R'(\alpha)-R_{\Q}'(\alpha)\|,\|X(\theta,\phi)-X_{\Q}(\theta, \phi)\|, \|M(\theta, \phi)-M_{\Q}(\theta, \phi)\|, \\
        \|M^\theta(\theta,\phi)-M_{\Q}^\theta(\theta,\phi)\|,
        \|M^\phi(\theta,\phi) - M_{\Q}^\phi(\theta,\phi)\|\leq \kappa.
    \end{align*}
    Moreover,
    \[
        \|R_{\Q}(\alpha)\|, \|R'_{\Q}(\alpha)\|, \|X_{\Q}(\theta, \phi)\|, \|M_{\Q}(\theta, \phi)\|, \|M_{\Q}^{\theta}(\theta,\varphi)\|, \|M_{\Q}^{\phi}(\theta,\varphi)\| \leq 1+\kappa.
    \]
\end{corr}
\begin{proof}
    The first statement is a direct application of \cref{lem:dist<kappa} and the second statement follows immediately after using \cref{lem:RaRalpha} and the triangle inequality.
\end{proof}

\begin{lem} \label{lem:A1AnB1Bn}
    For $1 \leq i \leq n$ let $(A_i,B_i)$ be pairs of real matrices, such that for each $i$ the dimensions of $A_i$ and $B_i$ are equal. Assume moreover that the products $A_1\cdots A_n$ and $B_1 \cdots B_n$ are well defined. Finally, assume that $\|A_i-B_i\|\leq \kappa$ and let $\delta_i\geq \max(\|A_i\|,\|B_i\|,1)$. Then it holds that
    $\|A_1\cdots A_n-B_1\cdots B_n\|\leq n\kappa\cdot \delta_1\cdots \delta_n$.
\end{lem}
Note, that this lemma can also be applied if some of the matrices are vectors or even scalars.
\begin{proof}
    We prove the statement by induction on $n$. For $n=1$ the inequality $\|A_1-B_1\|\leq \kappa \delta_1$ is true since $\|A_1-B_1\|\leq \kappa$ and $\delta_1\geq 1$.

    Now let $\alpha_n \coloneqq \|A_1\cdots A_n-B_1\cdots B_n\|$. We find, using the triangle inequality, submultiplicativity of the norm and $\|A_i\|,\|B_i\| \leq \delta_i$ that
    \[
        \alpha_{n} \leq \| A_1 \cdots A_{n-1} \| \cdot \| A_{n} - B_{n} \| + \alpha_{n-1} \cdot \|B_{n}\| \leq \delta_1 \cdots \delta_{n-1} \cdot \kappa + \alpha_{n-1} \cdot \delta_{n}.
    \]
    By induction it holds that $\alpha_{n-1} \leq (n-1)\kappa \cdot \delta_1\cdots \delta_{n-1} $, thus we obtain that 
    \[
        \alpha_{n} \leq \delta_1 \cdots \delta_{n-1} \cdot \kappa + (n-1)\kappa \cdot \delta_1\cdots \delta_{n-1} \cdot \delta_{n} = \kappa \cdot \delta_1\cdots \delta_{n-1} \cdot (1 + (n-1) \delta_n) \leq n\kappa \cdot \delta_1\cdots \delta_n,
    \]
    using that $1 + (n-1) \delta_n \leq n\delta_n$. 
\end{proof}

\subsection{Rational Global Theorem}

The rational version of \cref{thm:global} has the following form:

\begin{thm}[Rational Global Theorem] \label{thm:global_rational}
    Let $\PPP$ be a pointsymmetric convex polyhedron with radius $\rho =1$ and $\widetilde{\PPP}$ a $\kappa$-rational approximation. Let $\widetilde{S} \in \widetilde{\PPP}$. Further let $\epsilon>0$ and $\thetab_1,\phib_1,\thetab_2,\phib_2,\alphab \in \Q \cap [-4,4]$.
    Let $w\in\Q^2$ be a unit vector. Denote $\Mib \coloneqq M_{\Q}(\thetab_1, \phib_1)$, $ \Miib \coloneqq M_{\Q}(\thetab_2, \phib_2)$ as well as $\Mib^{\theta} \coloneqq M_{\Q}^\theta(\thetab_1, \phib_1)$, $\Mib^{\phi} \coloneqq M_{\Q}^\phi(\thetab_1, \phib_1)$ and analogously for $\Miib^{\theta}, \Miib^{\phi}$. Finally set
    \begin{align*}
        G^{\Q}& \coloneqq \langle R_{\Q}(\alphab) \Mib \widetilde{S},w \rangle - \epsilon\cdot\big(|\langle R_{\Q}'(\alphab)  \Mib \widetilde{S},w \rangle|+|\langle R_{\Q}(\alphab) \Mib^\theta \widetilde{S},w \rangle|+|\langle R_{\Q}(\alphab) \Mib^\phi \widetilde{S},w \rangle|\big) \\
        & \hspace{11cm}- 9\epsilon^2/2 - 4\kappa ( 1 + 3 \epsilon),\\
        H^{\Q}_P & \coloneqq \langle \Miib P,w \rangle + \epsilon\cdot\big(|\langle \Miib^\theta P,w \rangle|+|\langle  \Miib^\varphi P,w \rangle|\big) + 2\epsilon^2 + 3\kappa( 1+2\epsilon).
    \end{align*}
    If $G^{\Q}>\max_{P\in \widetilde{\PPP}} H^{\Q}_P$ then there does not exist a solution to Rupert's condition (\ref{eq:rupert_condition}) to $\PPP$ with 
    \[
    (\theta_1,\varphi_1,\theta_2,\varphi_2,\alpha) \in [\thetab_1\pm\epsilon,\phib_1\pm\epsilon,\thetab_2\pm\epsilon,\phib_2\pm\epsilon,\alphab\pm\epsilon].
    \]
\end{thm}

The proof relies on \cref{thm:global} as well as the following lemma which is a corollary of \cref{lem:A1AnB1Bn}:
\begin{lem} \label{lem:boundskappa}
    Let $\alpha, \theta, \phi \in [-4,4]$, $P\in \R^3$ with $\|P\| \leq 1$ and let $\widetilde{P}$ be a $\kappa$-rational approximation of $P$. Set $M = M(\theta, \phi)$ and $M_{\Q} = M_{\Q}(\theta, \phi)$, $M^\theta = M^\theta(\theta, \phi)$, $M^\theta_{\Q} = M^\theta_{\Q}(\theta, \phi)$, $M^\phi = M^\phi(\theta, \phi)$, $M^\phi_{\Q} = M^\phi_{\Q}(\theta, \phi)$ as well as $R = R(\alpha)$, $R_{\Q} = R_{\Q}(\alpha)$, $R' = R'(\alpha)$, $R'_{\Q} = R'_{\Q}(\alpha)$. Finally let $w \in \R^2$ with $\|w\| = 1$. Then:
    \begin{align}
        | \langle M P, w\rangle - \langle M_{\Q} \widetilde{P}, w\rangle | & \leq 3\kappa, \label{eq:boundskappa1} \\
        | \langle M^\theta P, w\rangle - \langle M^\theta_{\Q} \widetilde{P}, w\rangle | & \leq 3\kappa,\\
        | \langle M^\phi P, w\rangle - \langle M^\phi_{\Q} \widetilde{P}, w\rangle | & \leq 3\kappa,\\
        | \langle R M P, w\rangle - \langle R_{\Q} M_{\Q} \widetilde{P}, w\rangle | & \leq 4\kappa,\label{eq:boundskappa4} \\
        | \langle R' M P, w\rangle - \langle R'_{\Q} M_{\Q} \widetilde{P}, w\rangle | & \leq 4\kappa,\\
        | \langle R M^\theta P, w\rangle - \langle R_{\Q} M^\theta_{\Q} \widetilde{P}, w\rangle | & \leq 4\kappa,\\
        | \langle R M^\phi P, w\rangle - \langle R_{\Q} M^\phi_{\Q} \widetilde{P}, w\rangle | & \leq 4\kappa.
    \end{align}
\end{lem}
\begin{proof}
    The proof is a direct application of \cref{lem:A1AnB1Bn}. We will only show (\ref{eq:boundskappa1}) and (\ref{eq:boundskappa4}) as the other inequalities are proven in a completely analogous way. 
    
    Using the Cauchy-Schwarz inequality we obtain
    \[
        | \langle M P, w\rangle - \langle M_{\Q} \widetilde{P}, w\rangle | = | \langle M P - M_{\Q} \widetilde{P}, w\rangle | \leq \| MP - M_{\Q} \widetilde{P} \| \cdot \|w\|.
    \]
    Now observe that $\|w\|=1$ and $\| MP - M_{\Q} \widetilde{P} \| \leq 2 \kappa (1+\kappa)^2$ by \cref{lem:A1AnB1Bn} since $\|M_{\Q}\|, \|\widetilde{P}\| \leq 1 + \kappa$ by \cref{corr:kappa1kappa}. Then (\ref{eq:boundskappa1}) follows because $(1+\kappa)^2 < 3/2$.

    For (\ref{eq:boundskappa4}) it follows in the same way that 
    \[
        | \langle R M P, w\rangle - \langle R_{\Q} M_{\Q} \widetilde{P}, w\rangle | \leq \|R M P - R_{\Q} M_{\Q} \widetilde{P}\| \cdot \|w\|.
    \]
    Then \cref{lem:A1AnB1Bn} implies that $\|R M P - R_{\Q} M_{\Q} \widetilde{P}\| \leq 3 \kappa (1 + \kappa)^3$ and the statement follows because $(1 + \kappa)^3 < 4/3$.
\end{proof}

\begin{proof}[Proof of \cref{thm:global_rational}]
    Let $G, H_{P}$ be the quantities in \cref{thm:global}. Then observe that 
    \[
        G \geq G^{\Q} \quad \text{ and } \quad H^{\Q}_P \geq H_P
    \]
    by \cref{lem:boundskappa}. It follows that, under the assumption of the theorem,  $G \geq G^{\Q} > \max_{P} H^{\Q}_P \geq  \max_{P} H_P$, thus we can apply \cref{thm:global} and conclude. 
\end{proof}

\subsection{Rational Local Theorem}
Before we state the rational version of \cref{thm:local} we formulate the following $\kappa$-adapted version of $\epsilon$-spanning points (see \cref{sec:spanning}):
\begin{dfn} \label{def:ekspanning}
    Let $\theta, \phi \in \Q \cap [-4,4]$ and $M_{\Q} \coloneqq M_{\Q}(\theta, \phi)$. Three points $\widetilde{P}_1, \widetilde{P}_2, \widetilde{P}_3 \in \Q^3$ with $\|\widetilde{P}_1\|, \|\widetilde{P}_2\|, \|\widetilde{P}_3\| \leq 1+\kappa$ are called \emph{$\epsilon$-$\kappa$-spanning for $(\theta, \phi)$} if it holds that:
\begin{align*}
    \langle R(\pi/2) M_{\Q} \widetilde{P}_1,M_{\Q} \widetilde{P}_{2}\rangle > 2 \epsilon(\sqrt{2} + \epsilon) + 6\kappa,\\
    \langle R(\pi/2) M_{\Q} \widetilde{P}_2,M_{\Q} \widetilde{P}_{3}\rangle > 2 \epsilon(\sqrt{2} + \epsilon) + 6\kappa,\\
    \langle R(\pi/2) M_{\Q} \widetilde{P}_3,M_{\Q} \widetilde{P}_{1}\rangle > 2 \epsilon(\sqrt{2} + \epsilon) + 6\kappa.
\end{align*}
\end{dfn}

\begin{lem} \label{lem:ekspanningespanning}
    Let $P_1, P_2, P_3 \in \R^3$ with $\|P_i\| \leq 1$ and $\widetilde{P}_1, \widetilde{P}_2, \widetilde{P}_3 \in \Q^3$ be their $\kappa$-rational approximations. Assume that $\widetilde{P}_1, \widetilde{P}_2, \widetilde{P}_3$ are $\epsilon$-$\kappa$-spanning for some $\theta, \phi \in \Q \cap [-4,4]$, then $P_1, P_2, P_3$ are $\epsilon$-spanning for $\theta, \phi$.
\end{lem}

\begin{proof}
    Set $M \coloneqq M(\theta, \phi)$ and $M_{\Q} \coloneqq M_{\Q}(\theta, \phi)$. It suffices to show that  
    \begin{align}\label{eq:Rpi/2kappa}
        | \langle R(\pi/2) M P_1,M P_{2}\rangle - \langle R(\pi/2) M_{\Q} \widetilde{P}_1,M_{\Q} \widetilde{P}_{2}\rangle| \leq 6 \kappa,
    \end{align}
    since then 
    \[
        \langle R(\pi/2) M P_1,M P_{2}\rangle \geq \langle R(\pi/2) M_{\Q} \widetilde{P}_1,M_{\Q} \widetilde{P}_{2}\rangle - 6\kappa > 2 \epsilon(\sqrt{2} + \epsilon),
    \]
    analogously for the other two inequalities, and then \cref{lem:eps-spanning} concludes the proof.

    To see that inequality (\ref{eq:Rpi/2kappa}) holds true, note that the left-hand side can be written as 
    \[
        | P_2^t M^t R(\pi/2) M P_1 - \widetilde{P}^t_{2} M^t_{\Q} R(\pi/2) M_{\Q} \widetilde{P}_1|, 
    \]
    which is bounded by $5 \kappa (1 + \kappa)^4 \leq 6\kappa$ with \cref{lem:A1AnB1Bn}.
\end{proof}

Because of the appearing norms in the local theorem we have to deal with square roots which usually do not map $\Q$ to $\Q$. Depending on the side of the inequality which we wish to round, we will need versions of $\sqrt{x}$ which are bigger or smaller. For this reason we introduce  $\sqrt[+]{x}$ and $\sqrt[-]{x}$:

\begin{dfn} \label{def:sqrtplusminus}
    We call a function $x \mapsto \sqrt[+]{x} : \Q_{\geq 0} \to \Q$ \emph{upper-$\Q$-square-root} if $\sqrt[+]{x} \geq \sqrt{x}$ for all $x \in \Q_{\geq0}$. Similarly $\sqrt[-]{x} : \Q_{\geq 0} \to \Q$ is called \emph{lower-$\Q$-square-root} if $\sqrt[-]{x} \leq \sqrt{x}$ for all $x \in \Q_{\geq 0}$.
\end{dfn}

The rational version of \cref{thm:local} has the following form:
\begin{thm}[Rational Local Theorem] \label{thm:local_rational}
    Let $\PPP$ be a polyhedron with radius $\rho=1$ and $\widetilde{P}_i$ be a $\kappa$-rational approximation of $P_i \in \PPP$. Set $\widetilde{\PPP} = \{\widetilde{P}_i \text{ for } P_i \in \PPP \}$. Let $P_1, P_2, P_3, Q_1, Q_2, Q_3 \in \PPP$ be not necessarily distinct and assume that $P_1, P_2, P_3$ and $Q_1, Q_2, Q_3$ are congruent. 
    Let $\epsilon>0$ and $\thetab_1,\phib_1,\thetab_2,\phib_2,\alphab \in \Q \cap [-4,4]$. 
    Set $\Xib \coloneqq X_{\Q}(\thetab_1,\phib_1), \Xiib \coloneqq X_{\Q}(\thetab_2,\phib_2)$ as well as $\Mib \coloneqq M_{\Q}(\thetab_1,\phib_1), \Miib \coloneqq M_{\Q}(\thetab_2,\phib_2)$.
    Assume that there exist $\sigma_P, \sigma_Q \in \{0,1\}$ such that 
    \[
        (-1)^{\sigma_P} \langle \Xib,\widetilde{P}_i\rangle>\sqrt{2}\varepsilon + 3\kappa \quad \text{and} \quad
        (-1)^{\sigma_Q} \langle \Xiib , \widetilde{Q}_i\rangle>\sqrt{2}\varepsilon + 3\kappa, \tag{A$^{\Q}_\epsilon$}
    \]
    for all $i=1,2,3$.
    Moreover, assume that $\widetilde{P}_1,\widetilde{P}_2,\widetilde{P}_3$ are $\epsilon$-$\kappa$-spanning for $(\thetab_1,\phib_1)$ and that $\widetilde{Q}_1,\widetilde{Q}_2,\widetilde{Q}_3$ are $\epsilon$-$\kappa$-spanning for $(\thetab_2,\phib_2)$. Let $\sqrt[+]{x}$ and $\sqrt[-]{x}$ be upper- and lower-$\Q$-square-root functions, then set $\|Z\|_{+} \coloneqq \sqrt[+]{\|Z\|^2}$ and $\|Z\|_{-} \coloneqq \sqrt[-]{\|Z\|^2}$ for $Z \in \Q^n$.
    Finally, assume that for all $i = 1,2,3$ and any $\widetilde{Q}_j \in \widetilde{\PPP} \setminus \widetilde{Q}_i$ it holds that 
    \[
        \frac{\langle \Miib \widetilde{Q}_i,\Miib (\widetilde{Q}_i-\widetilde{Q}_j)\rangle - 10\kappa - 2 \epsilon ( \|\widetilde{Q}_i-\widetilde{Q}_j\|_{+} + 2 \kappa ) \cdot  (\sqrt{2}+\varepsilon)}{ \big(\|\Miib \widetilde{Q}_i\|_{+}+\sqrt{2} \varepsilon + 3\kappa \big) \cdot \big(\|\Miib(\widetilde{Q}_i-\widetilde{Q}_j)\|_{+}+2 \sqrt{2} \varepsilon + 6\kappa\big)} > \frac{\sqrt{5} \epsilon + \delta}{r}, \tag{B$^{\Q}_\epsilon$}
    \]
    for some $r >0$ such that $\min_{i=1,2,3}\| \Miib \widetilde{Q}_i \|_{-} > r + \sqrt{2} \epsilon + 3\kappa$ and for some $\delta \in \R$ with 
    \[
        \delta = \max_{i=1,2,3}\left\|R_{\Q}(\alphab) \Mib \widetilde{P}_i-\Miib \widetilde{Q}_i\right\|_{+}/2 + 3\kappa.
    \]
    Then there exists no solution to Rupert's problem $R(\alpha) M(\theta_1,\phi_1)\PPP \subset  M(\theta_2,\phi_2)\PPP^\circ$ with 
    \[
        (\theta_1, \phi_1, \theta_2, \phi_2, \alpha) \in [\thetab_1\pm\epsilon,\phib_1\pm\epsilon,\thetab_2\pm\epsilon,\phib_2\pm\epsilon,\alphab\pm\epsilon] \subseteq \R^5.
    \]
\end{thm}

In order to prove the Rational Local \cref{thm:local_rational}, we will show that all all assumed inequalities imply the corresponding assumptions of \cref{thm:local}. For that purpose we introduce and prove the following lemmas:

\begin{lem} \label{lem:boundskappa3}
    Let $P,Q \in \R^3$ with $\|P\|,\|Q\|\leq 1$ and $\widetilde{P},\widetilde{Q}$ some respective $\kappa$-rational approximations. Moreover, let $\alpha, \theta, \phi \in \R \in [-4,4]$ and set $X = X(\theta, \phi)$, $X_{\Q} = X_{\Q}(\theta, \phi)$ as well as $M = M(\theta, \phi)$, $M_{\Q} = M_{\Q}(\theta, \phi)$. Then
    \begin{align}
        |\langle X, P \rangle - \langle X_{\Q}, \widetilde{P} \rangle| & \leq 3 \kappa, \label{eq:boundskappa3.1}\\
        |\langle MP, MQ \rangle - \langle M_{\Q} \widetilde{P}, M_{\Q}\widetilde{Q} \rangle| & \leq 5 \kappa, \label{eq:boundskappa3.3}\\
        |\| M Q \| - \| M_{\Q}\widetilde{Q} \| | & \leq 3 \kappa.\label{eq:boundskappa3.2}
    \end{align}
\end{lem}

\begin{proof}
    Note that if $\|P\| \leq 1$ then $\|\widetilde{P}\| \leq 1 + \kappa$. Rewrite (\ref{eq:boundskappa3.1}) as $ | P^t X - \widetilde{P}^t X_{\Q} | $ and then observe that \cref{lem:A1AnB1Bn} yields the bound $2 \kappa (1 + \kappa)^2 \leq 3\kappa$.

    Inequality (\ref{eq:boundskappa3.3}) follows in a similar way by rewriting the left-hand side as $| Q^t M^t MP -  \widetilde{Q}^t M_{\Q}^t M_{\Q} \widetilde{P} |$, applying \cref{lem:A1AnB1Bn} and using that $4 \kappa (1+\kappa)^4 \leq 5\kappa$.

    Finally for (\ref{eq:boundskappa3.2}), observe that by the reverse triangle inequality it holds that $|\| M Q \| - \| M_{\Q}\widetilde{Q} \| |\leq \| MQ - M_{\Q} \widetilde{Q} \| $ and the latter is bounded by $2 \kappa (1 + \kappa)^2 \leq  3\kappa$ using \cref{lem:A1AnB1Bn}.
\end{proof}

\begin{corr} \label{corr:deltakappa}
    In the setting of  \cref{lem:boundskappa3} let additionally $\thetab, \phib \in \R \cap [-4,4]$ and set $\overline{M} = M(\thetab, \phib)$, $\overline{M}_{\Q} = M_{\Q}(\thetab, \phib)$. Then
    \[
    |\| R(\alpha) M P - \overline{M} Q\|- \| R_{\Q}(\alpha) M_{\Q} \widetilde{P} - \overline{M}_{\Q} \widetilde{Q}\| | \leq 6 \kappa.\label{eq:boundskappa3.4} 
    \]
\end{corr}
\begin{proof}
    Using $\big| \|a-b\| - \|c-d\|\big| \leq \|a-c\| + \|b-d\|$ we first bound the left-hand side by 
    \[
        \| R(\alpha) M P - R_{\Q}(\alpha) M_{\Q} \widetilde{P} \| + \| \overline{M} Q - \overline{M}_{\Q} \widetilde{Q}\|,
    \]
    then apply (\ref{eq:boundskappa3.2}) and the same proof as of (\ref{eq:boundskappa4}) to obtain the desired bound $ 3\kappa + 3\kappa = 6 \kappa$.
\end{proof}

\begin{corr} \label{lem:boundskappa4}
    In the setting of \cref{lem:boundskappa3}, let $\sqrt[+]{x}$ be an upper-$\Q$-square-root function and set $\|x\|_{+} \coloneqq \sqrt[+]{\|x\|^2}$. Set
    \[
        A =  \frac{\langle M P, M(P-Q)\rangle - 2 \epsilon  \|P-Q\| \cdot  (\sqrt{2}+\varepsilon)}{ \big(\| M P\|+\sqrt{2} \varepsilon \big) \cdot \big(\|M(P-Q)\|+2 \sqrt{2} \varepsilon\big)} 
    \]
    as well as 
    \[
        A_{\Q} =         \frac{\langle M_{\Q} \widetilde{P}, M_{\Q} (\widetilde{P}-\widetilde{Q})\rangle - 10\kappa - 2 \epsilon ( \|\widetilde{P}-\widetilde{Q}\|_{+} + 2 \kappa ) \cdot  (\sqrt{2}+\varepsilon)}{ \big(\| M_{\Q} \widetilde{P}\|_{+}+\sqrt{2} \varepsilon + 3\kappa \big) \cdot \big(\|M_{\Q}(\widetilde{P}-\widetilde{Q})\|_{+}+2 \sqrt{2} \varepsilon + 6\kappa\big)}.
    \]
    Assuming that $A_{\Q} > 0$, it then holds that $A \geq A_{\Q}$.
\end{corr}
\begin{proof}
    Note that $(\widetilde{P}-\widetilde{Q})/2$ is a $\kappa$-rational approximation of $(P-Q)/2$ and that $\|(P-Q)/2\| \leq 1$. Therefore, by \cref{lem:boundskappa3}'s inequality (\ref{eq:boundskappa3.3}) it holds that 
    \[
       \bigg| \Big \langle M P, M \frac{P-Q}{2} \Big \rangle - \Big \langle M_{\Q} \widetilde{P}, M_{\Q} \frac{\widetilde{P}-\widetilde{Q}}{2} \Big \rangle \bigg | \leq  5\kappa,
    \]
    and consequently we obtain that 
    \[
        \langle M P, M (P-Q) \rangle \geq \langle M P, M (\widetilde{P}-\widetilde{Q}) \rangle - 10\kappa.
    \]
    Further, the triangle inequality implies that 
    \[
        \|P-Q\| \leq \|\widetilde{P}-\widetilde{Q}\| + 2 \kappa \leq \|\widetilde{P}-\widetilde{Q}\|_{+} + 2 \kappa.
    \]
    It follows that the numerator of $A_{\Q}$ is at most the numerator of $A$ and that both are positive. 

    Finally, (\ref{eq:boundskappa3.2}) implies that 
    \[
        \| M P\| \leq \| M_{\Q} \widetilde{P}\| + 3\kappa \leq \|  M_{\Q} \widetilde{P}\|_{+} + 3\kappa, 
    \]
    and also that 
    \[
        \| M (P-Q)\| = 2\Big \| M \frac{P-Q}{2} \Big\| \leq 2 \Big \| M_{\Q} \frac{\widetilde{P} - \widetilde{Q}}{2} \Big\| + 6\kappa \leq \|  M_{\Q} (\widetilde{P} - \widetilde{Q})\|_{+} + 6\kappa.
    \]
    Therefore, the denominator of $A$ is at most the denominator of $A_{\Q}$. 
\end{proof}

\begin{proof}[Proof of \cref{thm:local_rational}]
    We shall show that we may apply \cref{thm:local}. 
    
    The inequalities in (A$^{\Q}_\epsilon$) imply (A$_\epsilon$) thanks to (\ref{eq:boundskappa3.1}) in \cref{lem:boundskappa3}. Similarly, (B$^{\Q}_\epsilon$) implies (B$_\epsilon$) by \cref{lem:boundskappa4}. Note that the bound on $r$ uses (\ref{eq:boundskappa3.2}) and the bound for $\delta$ uses \cref{corr:deltakappa}. The assumption that $\widetilde{P}_1,\widetilde{P}_2,\widetilde{P}_3$ are $\epsilon$-$\kappa$-spanning implies that $P_1,P_2,P_3$ are $\epsilon$-spanning by \cref{lem:ekspanningespanning}, and analogously for $Q_1,Q_2,Q_3$.
\end{proof}

\section{The algorithmic part of the proof} \label{sec:algoproof}

In \cref{cor:NOPbounds} of \cref{sec:noperthedron} we have proven that if $\NOP$ was Rupert then there would exists a solution to its Rupert's condition (\ref{eq:rupert_condition}) with  
\begin{align*}    
\theta_1,\theta_2&\in[0,2\pi/15] \subset [0,0.42], \\
\varphi_1&\in [0,\pi] \subset [0,3.15],\\
\varphi_2&\in [0,\pi/2] \subset [0,1.58],\\
\alpha &\in [-\pi/2,\pi/2] \subset [-1.58,1.58].
\end{align*}

Generally speaking, in order to prove that a solution for any solid $\PPP$ does not exist within any given five dimensional interval for $[\thetab_1, \phib_1, \thetab_2, \phib_2, \alphab] \pm \epsilon \subseteq \R^5$, we can establish 3 (recursive) options:
\begin{itemize}
    \item [(I)] Apply the (rational) global theorem (\cref{thm:global}),
    \item [(II)] Apply the (rational) local theorem (\cref{thm:local}),
    \item [(III)] Subdivide the interval into smaller intervals, and for each of them again apply one of (I)--(III).
\end{itemize}

To show that the Noperthedron is not Rupert, one can now think of creating a ``solution tree'' with each node corresponding to some interval and the root of the tree representing the initial interval from \cref{cor:NOPbounds}. 
If a node/interval can be proven to not contain a solution (using the global or the local theorems) then this node is a leaf. Otherwise, one might write the interval as a union of smaller intervals, each of which is then one of its children.

For $\NOP$ we have created such a tree algorithmically in the programming language \href{https://www.r-project.org/}{R} using floating point operations. 
On a good PC this took roughly 7 hours and used less than 4 GB of RAM. This tree has 585\,200 interior nodes and 18\,115\,245 leaf nodes, 17\,492\,530 of which can be solved using the global theorem and 622\,715 of which using the local theorem. It is stored in the form of a tabular database such that each of its rows corresponds to one node. Additionally, we store information about how the rational global/local theorems can be applied. 
The generation code and solution file can be found in \href{https://github.com/Jakob256/Rupert}{www.github.com/Jakob256/Rupert} with an excerpt of the latter displayed in \cref{tab:solution_tree} below. In \cref{sec:structure} we explain more details about the  structure of the solution tree.

Note that the floating point driven \href{https://www.r-project.org/}{R} code is not part of the proof of \cref{thm:main}: The algorithmic rigorous verification of the solution tree was programmed in \href{https://www.sagemath.org/}{SageMath}. This code checks that the rational versions of the global and local theorems can be applied to each leaf and that the union of them is the initial interval. This step is computationally more expensive and takes about 30 hours on the same PC. In \cref{sec:verification} we explain this code which is available on \href{https://github.com/Jakob256/Rupert}{www.github.com/Jakob256/Rupert}.

\subsection{Structure of the solution tree} \label{sec:structure}
As mentioned above, our tabular solution tree for the Noperthedron contains all necessary information needed for proving that this solid does not have Rupert's property. Each of its rows corresponds to a five-dimensional interval for a possible solution $(\theta_1, \phi_1, \theta_2, \phi_2, \alpha)$ with an indication which of the three options above (I), (II), (III) to apply. In the subdivision case (III) the data in the row points to the children of the node and in the cases (I) and (II) the row contains information on how to apply the global or local theorems to the interval. 
More precisely, the table has the following~columns:
\begin{enumerate}
    \item \textbf{ID}: Is a unique consecutive number for each row.
    \item \textbf{nodeType}: A number 1, 2 or 3, corresponding to the options (I), (II) or (III) above. 
    \item \textbf{nrChildren}: In the subdivision case (nodeType=3), the number of children this node has.
    \item \textbf{IDfirstChild}: In the case nodeType=3, the ID of the first child node. The children are stored as consequentive rows, therefore nrChildren and IDfirstChild are enough to recover~them. 
    \item \textbf{split}: When nodeType=3, for convenience, we store a variable which records how the interval is going to be split. There are the following options:
    \begin{align*}
        \text{split} =1 & \iff \text{only the interval for } \thetab_1 \text{ is split},\\
        \text{split} =2 & \iff \text{only the interval for } \phib_1 \text{ is split},\\
        \text{split} =3 & \iff \text{only the interval for } \thetab_2 \text{ is split},\\
        \text{split} =4 & \iff \text{only the interval for } \phib_2 \text{ is split},\\
        \text{split} =5 & \iff \text{only the interval for } \alphab \text{ is split},\\
        \text{split} =6 & \iff \text{all five intervals are split in half}.
    \end{align*}
    In other words, for $\mathrm{split} \in \{1,\dots,5\}$ the corresponding interval is divided in nrChildren equal parts, whereas if $\mathrm{split}=6$ then all intervals are halved resulting in $\mathrm{nrChildren}=2^5 = 32$. 
    \item \textbf{T1\_min}, \textbf{T1\_max}, \textbf{V1\_min}, \textbf{V1\_max}, \textbf{T2\_min}, \textbf{T2\_max}, \textbf{V2\_min}, \textbf{V2\_max}, \textbf{A\_min}, \textbf{A\_max} are the numerators of the boundaries for the interval of the node/row:
    \begin{align*}
        N \cdot \theta_1 & \in [T1_{\min}, T1_{\max} ], \\ 
        N \cdot\phi_1 & \in [V1_{\min}, V1_{\max} ], \\
        N \cdot\theta_2 & \in [T2_{\min}, T2_{\max} ], \\ 
        N \cdot\phi_2 & \in [V2_{\min}, V2_{\max} ], \\ 
        N \cdot\alpha & \in [A_{\min}, A_{\max} ],
    \end{align*}
    with $N$ set to 15\,360\,000.
    \item \textbf{S\_index}: The index of the vertex $S$ for applying the (rational) global theorem (\cref{thm:global_rational}), used if nodeType=1.
    \item \textbf{wx\_nominator}, \textbf{wy\_nominator}, \textbf{w\_denominator} are the numerators and common denominator of the components of the unit vector $w = (w_x, w_y) \in \Q^2$ in \cref{thm:global_rational} (present if nodeType=1).
    \item \textbf{P1\_index}, \textbf{P2\_index}, \textbf{P3\_index}, \textbf{Q1\_index}, \textbf{Q2\_index}, \textbf{Q3\_index}: The indices of the congruent triples $P_1,P_2,P_3,Q_1,Q_2,Q_3 \in \NOP$ from \cref{thm:local_rational}, used if nodeType=2.
    \item \textbf{r}, \textbf{sigma\_Q}: The numerator $r'$ of the rational number $r = r'/1000 \in \Q$ and $\sigma_Q \in \{0,1\}$ from \cref{thm:local_rational} (if nodeType=2).
\end{enumerate}

Note that all entries of the table are integers, therefore no rounding loss occurs when loading or storing it. 
When talking about \textit{indices of vertices} we refer to the ordering $\NOP = (P_0, \dots, P_{89})$ defined as
\begin{align} \label{eq:ordering}
    P_{k+15i+45 \ell}=(-1)^\ell R_z\left(\frac{2\pi k}{15}\right)C_{i+1}
\quad \text{ for } \quad i \in \{0,1,2\}, \; k \in \{0,\dots,14\}, \; \ell \in \{0,1\},
\end{align}
and $C_1, C_2, C_3 \in \Q^3$ given in \cref{sec:noperthedron}. Moreover, we always assume that $\sigma_{P} = 0$, so we only store~$\sigma_Q \in \{0,1\}$. 

A small excerpt of the solution tree is displayed in \cref{tab:solution_tree}. We note that stored as a \textit{csv} file it takes roughly 2.5 GB of storage space. A zipped version of it can be downloaded from \href{https://github.com/Jakob256/Rupert}{www.github.com/Jakob256/Rupert}.


\begin{table}[]
\vspace{-2cm}
\scriptsize
    \centering
    \Rotatebox{90}{
    \begin{tabular}{c|c|c|c|c|c|c|c|c|c|c|c|c|c|c|c|c|c|c|c|c|c|c|c|c|c|c}
ID & T & $\#$C & ID1C & sp & $\thetab_{1,{\min}}$ & $\thetab_{1,\max}$ & $\phib_{1,\min}$ & ${\phib}_{1, \max}$ & ${\thetab}_{2, \min}$ &${\thetab}_{2, \max}$ & ${\phib}_{2, \min}$ & ${\phib}_{2, \max}$ & ${\alphab}_{\min}$ & ${\alphab}_{\max}$ & $P_1$ & $P_2$ & $P_3$ & $Q_1$ & $Q_2$ & $Q_3$ & $r$ & $\sigma_Q$ & $w_x$ & $w_y$ & $w_d$ & $S$ \\
\hline
    0 & 3 & 4 & 1 & 1 & 0 & 64..00 & 0 & 48..00 & 0 & 64..00 & 0 & 24..00 & -24..00 & 24..00 &  &  &  &  &  &  &  &  &  &  &  &  \\
\hline
    1 & 3 & 30 & 5 & 2 & 0 & 16..00 & 0 & 48..00 & 0 & 64..00 & 0 & 24..00 & -24..00 & 24..00 &  &  &  &  &  &  &  &  &  &  &  &  \\
\hline
    2 & 3 & 30 & 46..67 & 2 & 16..00 & 32..00 & 0 & 48..00 & 0 & 64..00 & 0 & 24..00 & -24..00 & 24..00 &  &  &  &  &  &  &  &  &  &  &  &  \\
\hline
    3 & 3 & 30 & 94..77 & 2 & 32..00 & 48..00 & 0 & 48..00 & 0 & 64..00 & 0 & 24..00 & -24..00 & 24..00 &  &  &  &  &  &  &  &  &  &  &  &  \\
\hline
    4 & 3 & 30 & 14..51 & 2 & 48..00 & 64..00 & 0 & 48..00 & 0 & 64..00 & 0 & 24..00 & -24..00 & 24..00 &  &  &  &  &  &  &  &  &  &  &  &  \\
\hline
    5 & 3 & 4 & 35 & 3 & 0 & 16..00 & 0 & 16..00 & 0 & 64..00 & 0 & 24..00 & -24..00 & 24..00 &  &  &  &  &  &  &  &  &  &  &  &  \\
\hline
    6 & 3 & 4 & 70..31 & 3 & 0 & 16..00 & 16..00 & 32..00 & 0 & 64..00 & 0 & 24..00 & -24..00 & 24..00 &  &  &  &  &  &  &  &  &  &  &  &  \\
\hline
    7 & 3 & 4 & 10..67 & 3 & 0 & 16..00 & 32..00 & 48..00 & 0 & 64..00 & 0 & 24..00 & -24..00 & 24..00 &  &  &  &  &  &  &  &  &  &  &  &  \\
\hline
    \vdots & \vdots & \vdots & \vdots & \vdots & \vdots & \vdots & \vdots & \vdots & \vdots & \vdots & \vdots & \vdots & \vdots & \vdots & \vdots & \vdots & \vdots & \vdots & \vdots & \vdots & \vdots & \vdots & \vdots & \vdots & \vdots & \vdots \\
\hline
    \vdots & \vdots & \vdots & \vdots & \vdots & \vdots & \vdots & \vdots & \vdots & \vdots & \vdots & \vdots & \vdots & \vdots & \vdots & \vdots & \vdots & \vdots & \vdots & \vdots & \vdots & \vdots & \vdots & \vdots & \vdots & \vdots & \vdots \\
\hline
    91 & 1 &  &  &  & 0 & 80..00 & 0 & 80..00 & 80..00 & 16..00 & 80..60 & 16..20 & -23..40 & -22..80 &  &  &  &  &  &  &  &  & 53..73 & 15..64 & 16..45 & 39 \\
\hline
    92 & 1 &  &  &  & 0 & 80..00 & 80..00 & 16..00 & 0 & 80..00 & 0 & 80..60 & -24..00 & -23..40 &  &  &  &  &  &  &  &  & 98..92 & 35..15 & 10..33 & 37 \\
\hline
    \vdots & \vdots & \vdots & \vdots & \vdots & \vdots & \vdots & \vdots & \vdots & \vdots & \vdots & \vdots & \vdots & \vdots & \vdots & \vdots & \vdots & \vdots & \vdots & \vdots & \vdots & \vdots & \vdots & \vdots & \vdots & \vdots & \vdots \\
\hline
    245 & 2 &  &  &  & 0 & 20..00 & 0 & 20..00 & 0 & 20..00 & 0 & 20..40 & -22..20 & -22..80 & 30 & 31 & 38 & 79 & 80 & 87 & 955 & 1 &  &  &  &  \\
\hline
    246 & 1 &  &  &  & 0 & 20..00 & 0 & 20..00 & 0 & 20..00 & 20..40 & 40..80 & -23..60 & -22..20 &  &  &  &  &  &  &  &  & 71..05 & 19..88 & 20..37 & 39 \\
\hline
    247 & 1 &  &  &  & 0 & 20..00 & 0 & 20..00 & 0 & 20..00 & 20..40 & 40..80 & -22..20 & -22..80 &  &  &  &  &  &  &  &  & 71..05 & 19..88 & 20..37 & 39 \\
\hline
    248 & 2 &  &  &  & 0 & 20..00 & 0 & 20..00 & 20..00 & 40..00 & 0 & 20..40 & -23..60 & -22..20 & 30 & 31 & 38 & 79 & 80 & 87 & 955 & 1 &  &  &  &  \\
\hline
    249 & 2 &  &  &  & 0 & 20..00 & 0 & 20..00 & 20..00 & 40..00 & 0 & 20..40 & -22..20 & -22..80 & 30 & 31 & 38 & 79 & 80 & 87 & 955 & 1 &  &  &  &  \\
\hline
    \vdots & \vdots & \vdots & \vdots & \vdots & \vdots & \vdots & \vdots & \vdots & \vdots & \vdots & \vdots & \vdots & \vdots & \vdots & \vdots & \vdots & \vdots & \vdots & \vdots & \vdots & \vdots & \vdots & \vdots & \vdots & \vdots & \vdots \\
\hline
    18..44 & 1 &  &  &  & 48..00 & 64..00 & 46..00 & 48..00 & 48..00 & 64..00 & 22..80 & 24..00 & 22..80 & 24..00 &  &  &  &  &  &  &  &  & 33..40 & -14..51 & 14..49 & 78 \\
    \end{tabular}
    }
    \caption{Excerpt from the solution tree for $\NOP$.}
    \label{tab:solution_tree}
\end{table}

\subsection{Verification} \label{sec:verification}
In order to verify that the solution tree indeed proves \cref{thm:main} we must check the following assertions:
\begin{enumerate}
    \item[(A)] The union of all leaf intervals (rows with nodeType=1 or 2) contains the interval in \cref{cor:NOPbounds},
    \item[(B)] For each row with nodeType=1 one can apply \cref{thm:global_rational} and for each row with nodeType=2 one can apply \cref{thm:local_rational}.
\end{enumerate}

In order to check (A) we first verify that for each node the union of all its children contains its interval. Further we check that all nodes are either leafs (nodeType=1 or 2) or interval nodes (nodeType=3), and that the root contains the interval in \cref{cor:NOPbounds}. With the fact that the ID of any child is always strictly larger than the ID of the parent, this proves (A).

The verification of (B) splits into two cases: nodeType=1 or nodeType=2. They are explained in \cref{sec:veriglob} and \cref{sec:veriloc}. For each of them $\widetilde{\NOP}$ is defined as the following truncation of all vertices of $\NOP$ to 16 decimal places and is therefore a $\kappa$-rational approximation:
\[
    \widetilde{\NOP} \coloneqq \lfloor10^{16} \cdot \NOP \rfloor /10^{16}.
\]
The indices correspond to points via the ordering in (\ref{eq:ordering}). Both cases start with the following two operations:
\begin{enumerate}
\item Set $N=15\,360\,000$ and compute 
    \begin{align*}
    \thetab_1 & = (T1_{\min} + T1_{\max} )/(2N), \\ 
    \phib_1 & = (V1_{\min} + V1_{\max} )/(2N), \\
    \thetab_2 & = (T2_{\min} + T2_{\max} )/(2N), \\ 
    \phib_2 & = (V2_{\min} + V2_{\max} )/(2N), \\ 
    \alphab & = (A_{\min} + A_{\max} )/(2N).
\end{align*}
Check that $\thetab_i, \phib_i, \alphab \in [-4,4]$.
\item Set $\epsilon = \max(T1_{\max} - T1_{\min}, V1_{\max} - V1_{\min},T2_{\max} - T2_{\min}, V1_{\max} - V1_{\min}, A_{\max} - A_{\min})/(2N)$.
\end{enumerate}

\subsubsection{Verification of the global theorem} \label{sec:veriglob}

The verification of the rational global theorem (\cref{thm:global_rational}) is relatively straightforward. Namely, for every row in the solution tree which has noteType=1 we must perform the following operations (additionally to 1. and 2. above):
\begin{enumerate}
    \setcounter{enumi}{2}
    \item Set $w = ($wx\_numerator/w\_denominator, wy\_numerator/w\_denominator$) \in \Q^2$ and check that $\|w\|^2 = w_1^2 + w_2^2=1$.
    \item Set $\widetilde{S}$ to be the rationally approximated point corresponding to the index given in column~S\_index.
    \item Compute $\Mib$, $\Miib$, $\Mib^\theta$, $\Mib^\phi$, $\Miib^\theta$, $\Miib^\phi \in \Q^{2 \times 3}$ and $R_{\Q}(\alphab), R'_{\Q}(\alphab) \in \Q^{2 \times 2}$ as in \cref{thm:global_rational}.
    \item Compute $G^{\Q} \in \Q$ and $H_{P}^{\Q} \in \Q$ for all $P \in \widetilde{\NOP}$ as in \cref{thm:global_rational}.
    \item Verify that $G^{\Q} > \max H_{P}^{\Q}$.
\end{enumerate}

Note that all steps only use rational numbers, thus in the computer algebra system \href{https://www.sagemath.org/}{SageMath} no rounding loss or precision errors can occur.

\subsubsection{Verification of the local theorem} \label{sec:veriloc}

For the verification of the rational local \cref{thm:local_rational} for all rows of the solution tree with nodeType=2 we have to perform more work.

First, recall that in \cref{def:sqrtplusminus} and \cref{thm:local_rational} we define/use the upper- and lower-square root functions $\sqrt[+]{x}, \sqrt[-]{x} \colon \mathbb{Q}_{\geq 0} \to \mathbb{Q}_{\geq 0}$ given via
    \[
        \sqrt[-]{x} \leq \sqrt{x} \leq \sqrt[+]{x}.
    \]
    We decided to implement the following functions since they allow a precise control on the size of the denominator/accuracy and at the same time can be relatively efficiently implemented: 
    
    First, set $\sqrt[-]{0}=\sqrt[+]{0}=0$. For any $x>0$, there exists a unique $a\in \mathbb{Z}$ such that 
    \[
        x \cdot 10^{2a}\in [10^{20},10^{22}),
    \]
    and then there exists a unique $b\in \mathbb{Z}^+$ such that 
    \[
        b^2\leq x\cdot 10^{2a} < (b+1)^2.
    \]
    Observe that $b/10^a \leq \sqrt{x}$, so we define 
    \begin{align} \label{eq:oursqrts}
        \sqrt[-]{x} \coloneqq \frac{b}{10^a} \quad \text{and} \quad \sqrt[+]{x} \coloneqq \frac{1}{\sqrt[-]{1/x}}.
    \end{align}
    We then also define 
    \[
        \|Z\|_{-} \coloneqq \sqrt[-]{\|Z\|^2} \quad \text{ and } \quad \|Z\|_{+} \coloneqq \sqrt[+]{\|Z\|^2}
    \]
    for any $Z \in \Q^{n}$, so that $\|Z\|_{-},\|Z\|_{+} \in \Q$ and $\|Z\|_{-} \leq \|Z\| \leq \|Z\|_{+}$. \\

    Similarly to the global theorem verification after the steps 1. and 2. above, we perform the following steps for each row with nodeType=2:

    \begin{enumerate}
    \setcounter{enumi}{2}
    \item Check that $P_1,P_2,P_3$ and $Q_1,Q_2,Q_3$ are congruent by applying \cref{lem:congruent} (verifying that $P^tP = Q^tQ$ and $\det(Q) \neq 0$).
    \item Compute $\Mib$, $\Miib \in \Q^{2 \times 3}$, $\Xib, \Xiib \in \Q^{3}$ and $R_{\Q}(\alphab) \in \Q^{2 \times 2}$ as in \cref{thm:global_rational}.
    \item Set $\sigma_Q =$ sigma\_Q and check that 
    \[
        \langle \Xib, \widetilde{P}_i\rangle>\sqrt{2}\varepsilon + 3 \kappa \quad \text{ and } \quad (-1)^{\sigma_Q} \langle \Xiib,\widetilde{P}_i\rangle>\sqrt{2}\varepsilon +3 \kappa.
    \]
    \item Check that $\widetilde{P}_1,\widetilde{P}_2,\widetilde{P}_3$ are $\epsilon$-$\kappa$-spanning for $(\thetab_1, \phib_1)$ and that $\widetilde{Q}_1,\widetilde{Q}_2,\widetilde{Q}_3$ are $\epsilon$-$\kappa$-spanning for $(\thetab_2, \phib_2)$ by verifying the inequalities in \cref{def:ekspanning}:
        \begin{align*}
                \langle R(\pi/2) \Mib \widetilde{P}_1, \Mib \widetilde{P}_{2}\rangle > 2 \epsilon(\sqrt{2} + \epsilon) + 6\kappa,\\
                \langle R(\pi/2) \Mib \widetilde{P}_2, \Mib \widetilde{P}_{3}\rangle > 2 \epsilon(\sqrt{2} + \epsilon) + 6\kappa,\\
                \langle R(\pi/2) \Mib \widetilde{P}_3, \Mib \widetilde{P}_{1}\rangle > 2 \epsilon(\sqrt{2} + \epsilon) + 6\kappa,
        \end{align*}
        and 
            \begin{align*}
                \langle R(\pi/2) \Miib \widetilde{Q}_1, \Miib \widetilde{Q}_{2}\rangle > 2 \epsilon(\sqrt{2} + \epsilon) + 6\kappa,\\
                \langle R(\pi/2) \Miib \widetilde{Q}_2, \Miib \widetilde{Q}_{3}\rangle > 2 \epsilon(\sqrt{2} + \epsilon) + 6\kappa,\\
                \langle R(\pi/2) \Miib \widetilde{Q}_3, \Miib \widetilde{Q}_{1}\rangle > 2 \epsilon(\sqrt{2} + \epsilon) + 6\kappa.
        \end{align*}
    \item Set $r = r'/1000 \in \Q$ with $r'$ taken from the r column and define
        \[
            \delta = \max_{i=1,2,3}\left\|R_{\Q}(\alphab) \Mib \widetilde{P}_i-\Miib \widetilde{Q}_i\right\|_{+}/2 + 3\kappa \in \Q.
        \]
    \item Check that $\| \Miib \widetilde{Q}_i \|_{-} > r + \sqrt{2} \epsilon + 3\kappa$ for $i=1,2,3$.
    \item Finally verify that 
    \[
        \frac{\langle \Miib \widetilde{Q}_i,\Miib (\widetilde{Q}_i-\widetilde{Q}_j)\rangle - 10\kappa - 2 \epsilon ( \|\widetilde{Q}_i-\widetilde{Q}_j\|_{+} + 2 \kappa ) \cdot  (\sqrt{2}+\varepsilon)}{ \big(\|\Miib \widetilde{Q}_i\|_{+}+\sqrt{2} \varepsilon + 3\kappa \big) \cdot \big(\|\Miib(\widetilde{Q}_i-\widetilde{Q}_j)\|_{+}+2 \sqrt{2} \varepsilon + 6 \kappa\big)} > \frac{\sqrt{5} \epsilon + \delta}{r},
    \]
    for all $i = 1,2,3$ and any $\widetilde{Q}_j \in \widetilde{\NOP} \setminus \widetilde{Q}_i$.
\end{enumerate}

Note that $\sqrt{2}$ in the steps 5., 6., 8. and 9. can be approximated by $142/100 = 1.42 > \sqrt{2}$ making the inequalities stricter and the expressions rational. Similarly, we approximate $\sqrt{5}$ in step 9. by $224/100 = 2.24 > \sqrt{5}$, so that also this condition becomes a statement on purely rational numbers. It follows that the only non-rational condition to check is step 3., namely that 
\(
P^tP = Q^tQ
\)
for $P = (P_1|P_2|P_3) \in \overline{\Q}^{3 \times 3}$, $Q = (Q_1|Q_2|Q_3)\in \overline{\Q}^{3 \times 3}$ and that $\det(Q) \neq 0$. Note that these expressions are, however, over algebraic numbers and \href{https://www.sagemath.org/}{SageMath} can prove them relatively~quickly. 

\bigskip

We implemented the verification of the solution tree integrity as well as the procedures of \cref{sec:veriglob} and \cref{sec:veriloc} in \href{https://www.sagemath.org/}{SageMath}. The total run time of this code is roughly 30h on a good PC which roughly consists of 1h for the solution tree integrity check, 16h of verification of rational global theorems and 10h of rational local theorems verification. 
This finally finishes the proof of \cref{thm:main}.

\section{Rupert but not locally Rupert} \label{sec:rupertnotlocal}
When proving that a polyhedron is not Rupert the most complicated (and most interesting) part is excluding ``local'' solutions (see also \cref{sec:discussion}) -- the main use case of \cref{thm:local}. The following definition formalizes this notion\footnote{Note that the same concept was already introduced in \cite{Scott2022} with a slightly different formulation.}: 
\begin{dfn}
    A convex pointsymmetric polyhedron is \emph{locally Rupert} if for every $\varepsilon>0$ there exists a solution $(\theta_1,\varphi_1,\theta_2,\varphi_2,\alpha)$ to Rupert's condition (\ref{eq:rupert_condition}) with
    \(
    |\theta_1-\theta_2|,|\varphi_1-\varphi_2|,|\alpha|\leq\varepsilon.
    \)
\end{dfn}

Obviously, being locally Rupert implies Rupert's property and looking at the solutions of previously known polyhedra, a straightforward question is whether Rupertness actually also implies local Rupertness: Maybe any ``non-local'' solution can be moved continuously and become ``local''. If this was true, a proof that the Noperthedron is not Rupert would be easier, since one would only need to rule out local solutions; accounted for symmetries accordingly, in some sense, one would only need to exclude a two-dimensional search space instead of the five-dimensional one. However, as we shall prove in \cref{thm:main2}, the aforementioned question has a negative answer: We can construct a convex polyhedron, named \emph{Ruperthedron} ($\RUP$), which is Rupert but not locally Rupert. Similarly to $\NOP$, we start by setting three points $D_1,D_2,D_3\in \mathbb{Q}^3$ to be:
\[
    D_1\coloneqq
        \frac{1}{5861}
        \begin{pmatrix} 
        {3939} \\ 0 \\ {4340} 
        \end{pmatrix},
    \qquad
    D_2\coloneqq \frac{1}{10^4}
        \begin{pmatrix} 
        7855 \\ 4178 \\ 4484
        \end{pmatrix},
    \qquad
    D_3\coloneqq
        \frac{1}{10^4}
        \begin{pmatrix} 
        9526 \\ 2057 \\ 1102
        \end{pmatrix}.
\]
Note that $1=\|D_1\|>\|D_2\|>\|D_3\|$. Then, as before, we define the Ruperthedron by the action of $\mathcal{C}_{30}$ on $D_1,D_2,D_3$:
\[
    \RUP \coloneqq \mathcal{C}_{30} \cdot D_1 \cup \mathcal{C}_{30}\cdot D_2\cup \mathcal{C}_{30} \cdot D_3.
\]

\begin{figure}
    \centering
    \includegraphics[width=0.5\linewidth]{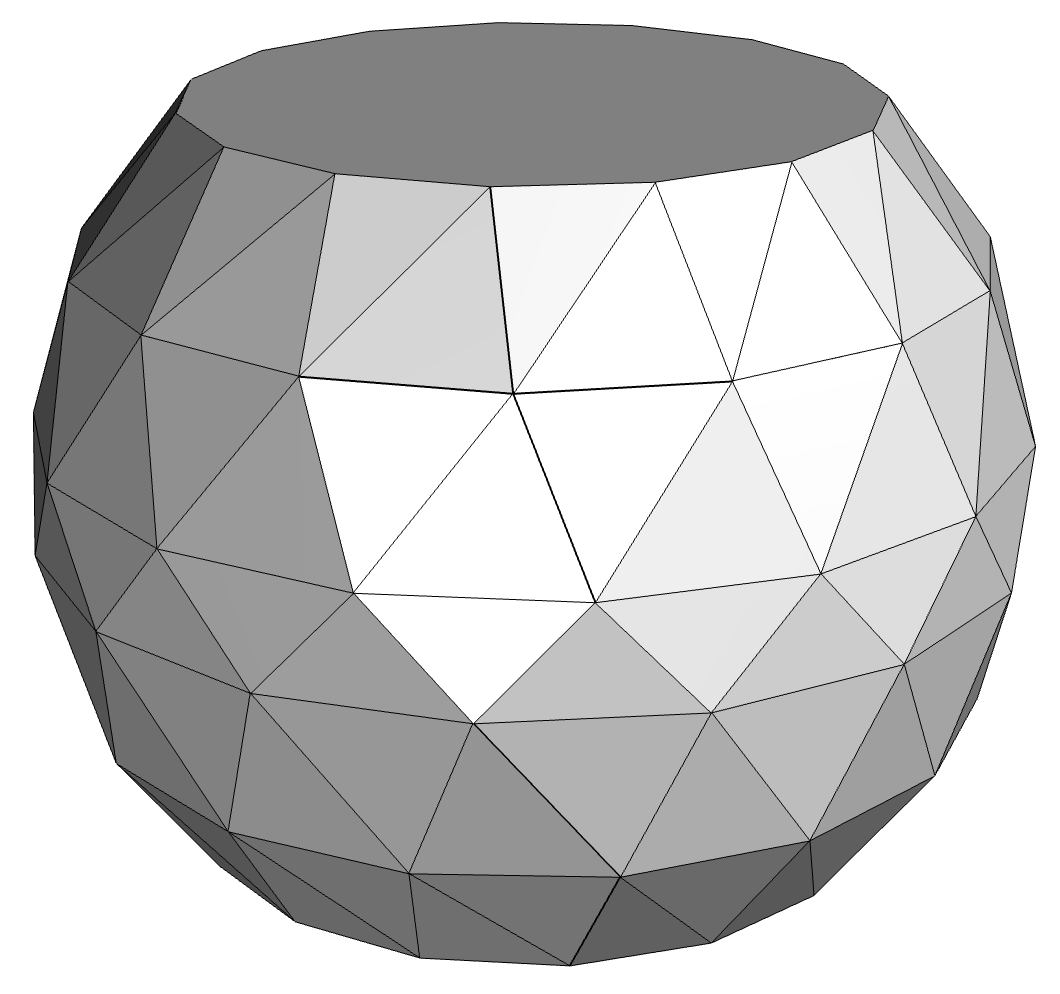}
    \caption{The Ruperthedron}
    \label{fig:ruperthedron}
\end{figure}

Like the Noperthedron, $\RUP$ has 90 vertices, is pointsymmetric with respect to the origin and all vertices have Euclidean norm of at most 1. An image of the Ruperthedron is displayed in Figure~\ref{fig:ruperthedron}.\\

Our main result of this section reads as follows:
\begin{thm} \label{thm:main2}
    The Ruperthedron has the following two properties:
    \begin{itemize}
        \item [1)] $\RUP$ is Rupert and its Nieuwland number\footnote{For the definition of Nieuwland number we refer to \cite{StYu23}.} is at least $1.003$.
        \item [2)] There does not exist a solution to Rupert's condition of $\RUP$ with 
        $$|\theta_1-\theta_2|,|\varphi_1-\varphi_2|,|\alpha|\leq 0.0006 \eqqcolon \omega.$$
    \end{itemize}
    In particular, $\RUP$ is Rupert but not locally Rupert.
\end{thm}

\begin{proof}
    A solution for $\RUP$ can be found using the methods described in \cite{StYu23}. Specifically, it can be checked using a computer algebra software that
    \[
        (\theta_1, \phi_1, \theta_2, \phi_2, \alpha)=(0.29,0.29,0.02,2.27,-1.02)
    \]
    is a solution for $\RUP$'s Rupert's property and that it has Nieuwland number greater than $1.003$. This proves (1).

    Using the same symmetries described in \cref{lem:symmetries}, it is clear that we can assume $\theta_1\in[0,2\pi/15], \varphi_1\in[0,\pi/2]$ and $\alpha \in [-\pi/2, \pi/2)$.
    We define the sequence of points $A_k \coloneqq \omega\cdot k \in \R$ for $k \in \mathbb{N}$. Note, that $A_{699}>2\pi/15$ and $A_{2618}>\pi/2$. We further define the five-dimensional closed cubes $B_{i,j}\subset \mathbb{R}^5$ for $i=0,\dots,699$, $j=0,\dots, 2618$ by
    \[
        B_{i,j} \coloneqq [A_i\pm\omega, A_j\pm\omega, A_i\pm\omega, A_j\pm\omega, 0\pm\omega].
    \]
    If a solution $\Psi = (\theta_1, \phi_1, \theta_2, \phi_2, \alpha)$ would exist as stated in (2) then, for some $0\leq i\leq 699$ and some $0\leq j\leq 2618$, it would hold that $\Psi \in B_{i,j}$. 
    
    However, similarly to the solution tree for \cref{thm:main}, we can construct regions $C_k\subset \mathbb{R}^5$ for $0\leq k \leq 3534$ such that each $B_{i,j}$ is a subset of at least one $C_k$, and such that it is possible to prove using the rational local \cref{thm:local_rational} that there does not exist a solution $\Psi \in C_k$ for all $k$. The definition of $(C_k)_{0 \leq k \leq 3534}$, the list which $B_{i,j}$  lies within which $C_k$ and the verification code are stored on \href{https://github.com/Jakob256/Rupert}{www.github.com/Jakob256/Rupert}.
\end{proof}

The structure of the solution tree for \cref{thm:main2} is similar to the one described in \cref{sec:structure} and \cref{tab:solution_tree}, but smaller since for the verification we only apply the rational local theorem. 
Moreover, this table can be simplified with the assumptions that $P_i=Q_i$ for $i=1,2,3$ (i.e. $L = \id$) and $\sigma_P = \sigma_Q = 0$. Further, $\thetab_1 = \thetab_2$, $\phib_1 = \phib_2$ and $\alphab = 0$ always in these applications of the local theorem.
The verification of this solution tree is also much faster than the one for \cref{thm:main}: On the same machine it can be done in less than 10 minutes. 

We note that the Ruperthedron is a natural candidate for disproving Conjecture 2 in \cite{Scott2022}: We believe that its dual (for any reasonable definition of this term) is locally Rupert but leave the proof for future work.

\section{Discussion} \label{sec:discussion}
In this concluding section we motivate our construction of $\NOP$ and explain how it was found. We also sketch possible directions for further research.

\subsection{Origin of the Noperthedron}

In \cite[Conjecture 2]{StYu23} we conjectured that the Rhombicosidodecahedron (\textbf{RID} in short) does not have Rupert's property. We still believe this conjecture, however the methods of this paper do not suffice to prove it. More precisely, while disproving ``global'' solutions for \textbf{RID} works in the same way as for \textbf{NOP}, using \cref{thm:global}, we cannot exclude all ``local'' solutions with \cref{thm:local}. 

Recall that the Rhombicosidodecahedron can be defined as the set of all even permutations of $(\pm1,\pm1,\pm\Phi^3)$, $(\pm\Phi^2,\pm\Phi,\pm2\Phi)$ and $(\pm(2+\Phi),0,\pm\Phi^2)$ with $\Phi = (1 + \sqrt{5})/2$ the golden ratio. One direction where the local theorem cannot be applied is the ``top view'', i.e. $\thetab_1 = \phib_1 = \thetab_2 = \phib_2 = \alphab = 0$. The reason is that, viewed from this direction, all projected points $\Miib Q_i$ on the convex hull either come from vertices normal to $\overline{X}_1$ (violating (A$_\epsilon$)) or there exist $Q_i \neq Q_j$ such that $\Miib Q_i = \Miib Q_j$ (violating condition (B$_\epsilon$)). 
Moreover, we observed that in approximately $2\%$ of random directions $(\thetab, \phib)$ the local theorem cannot be applied for $\RID$. 

This example shows: \emph{If one wants to find a solid without Rupert's property, then either one needs a new method for disproving such local highly non-generic projections, or one needs a solid where such projections do not exist.}

We note that the region around the direction $\thetab_1 = \phib_1 = \thetab_2 = \phib_2 = \alphab = 0$ for \textbf{RID} can be handled with a polynomial system of inequalities constructed similarly to \cite[\S3.3]{StYu23}, however there are even more nasty regions for the \textbf{RID} which are even more difficult to tackle. This is why we decided to drop the Rhombicosidodecahedron and look to other solids.\\

We searched through a database of polyhedra (\href{https://dmccooey.com/polyhedra/}{www.dmccooey.com/polyhedra/}) for solids where the local theorem can be universally applied for all ``local'' solutions. Of particular interest was the solid \href{https://dmccooey.com/polyhedra/132Pentagons.html}{132-Pentagon Polyhedron}, a polyhedron with an icosahedral symmetry consisting of 132 Pentagons. 
For this solid one can prove that it does not have local solutions using \cref{thm:local}, except for one direction. For this particular region we could actually exclude Rupert's property using polynomial solving techniques, however we were not satisfied with this solution. Moreover, handling the icosahedral symmetry for reducing the search space like in \cref{lem:symmetries} is rather~cumbersome. \\

Therefore, we decided to construct our own polyhedron with all desired properties: 
\begin{enumerate}
    \item[(1)] Obviously, the polyhedron must not have Rupert's property. 
    \item[(2)] The polyhedron must be convex and pointsymmetric.
    \item[(3)] For each projection direction $(\thetab, \phib)$ the local \cref{thm:local} must be applicable.
    \item[(4)] It should have as many symmetries as possible while having not too many vertices. Moreover, the symmetries should be as simple as possible.
\end{enumerate}

The pointsymmetric condition in (2) is essential as it reduces the search space from $\R^7$ to $\R^5$, as explained in \cite[Proposition 2]{StYu23}. The point (3) means that the local theorem may be applied for each $(\thetab, \phib,\thetab, \phib,0)$, i.e. there exist the triples $P_i, Q_i$, $i=1,2,3$ with the desired properties. In particular, this means that when projecting the polyhedron from direction $(\thetab, \phib)$ there always exist $\epsilon$-spanning points $P_1,P_2,P_3$ that are not projected to the same points as some other vertices. In particular, it follows that the polyhedron must not have a mirror symmetry.\footnote{We remark that mirror symmetries are connected to prism and polygonal sections and consequently to local Rupertness via the work done in~\cite{Scott2022}.} 

Evidently, the amount of symmetries reduces the search space, for instance, as mentioned in \cref{sec:intro} the 15-fold symmetry alone simplifies the search by a factor of 225. Having simple symmetries is an interesting condition: Famously, there are not that many finite groups of $\text{SO}_3(\R)$ and the symmetries of the octahedron and icosahedron are worthless for us, as they have an unbreakable mirror symmetry. 
Anything related to prisms (cyclic groups) is also not possible, as they have an intrinsic mirror symmetry as well. Antiprismas with an even number of points on one side are also not useful, as they are not point symmetric. This leaves antiprismas with an odd number of points on one side, however, they still have a mirror symmetry. 
Combining such different antiprismas with the same group $\mathcal{C}_n$, one can break this mirror symmetry while keeping $\mathcal{C}_n$. This motivates stacking multiple antiprismas with an odd number of vertices on top of each other. We chose $n=15$ and three antiprisms because this gave a good compromise between the amount of vertices and size for the interval to check (\cref{cor:NOPbounds}).

Finally, we ran a search for generator points $C_1, C_2, C_3$ (like in \cref{sec:noperthedron}) checking whether the local theorem can be applied from all directions to the resulting polyhedron. For the obtained candidates we tried to find solutions to Rupert's problem with our algorithm \cite[\S3.1]{StYu23}; if this was not possible in reasonable time, we checked whether the combined application of the global and local theorems can be used to disprove Rupert's property for the solid. 
Ultimately, we selected the solid which indeed was provably not Rupert and produced the smallest solution tree. This search took several weeks of continuous computing. 

\subsection{Open Problems}
Finally, we want to mention our favorite picks of open problems regarding Rupert's property:
\begin{itemize}
    \item Are there other ways to disprove the existence of local solutions? That is, without the method utilized in \cref{lem:langles} and \cref{thm:local}, or setting up an algebraic system of inequalities.  
    \item Is there a way to prove that a solid does not have Rupert's property without splitting $\R^5$ into millions of small regions and solving each separately? 
    \item Find a method to prove that a solution to Rupert's property is (locally) optimal, i.e. it has the (locally) the largest Nieuwland number. More generally, how to prove the conjectured Nieuwland numbers? It is, for instance, still open to prove that the Octahedron has Nieuwland number $3\sqrt{2}/4$ or that the Dodecahedron and Icosahedron both have Nieuwland number $\nu \approx 1.0108$, a root of
    \(
        P(x) = 2025x^8 - 11970x^6 + 17009x^4 - 9000x^2 + 2000
    \).
\end{itemize}

\bibliographystyle{plain}
\bibliography{bib} 

@article {StYu23,
    AUTHOR = {Steininger, Jakob and Yurkevich, Sergey},
     TITLE = {An algorithmic approach to {R}upert's problem},
   JOURNAL = {Math. Comp.},
  FJOURNAL = {Mathematics of Computation},
    VOLUME = {92},
      YEAR = {2023},
    NUMBER = {342},
     PAGES = {1905--1929},
      ISSN = {0025-5718},
   MRCLASS = {51 (52)},
  MRNUMBER = {4570346},
       DOI = {10.1090/mcom/3831},
       URL = {https://doi.org/10.1090/mcom/3831},
}

@book{Wallis1685,
  author    = {John Wallis},
  title     = {De Algebra Tractatus; Historicus \& Practicus},
  year      = {1685},
  booktitle = {Opera Mathematica},
  volume    = {2},
  publisher = {E Theatro Sheldoniano},
  address   = {Oxoniae},
  year      = {1693}
}

@article {Scriba68,
    AUTHOR = {Scriba, Christoph J.},
     TITLE = {Das {P}roblem des {P}rinzen {R}uprecht von der {P}falz},
   JOURNAL = {Praxis Math.},
  FJOURNAL = {Praxis der Mathematik},
    VOLUME = {10},
      YEAR = {1968},
    NUMBER = {9},
     PAGES = {241--246},
      ISSN = {0032-7042},
   MRCLASS = {01A45 (50B30)},
  MRNUMBER = {497615},
MRREVIEWER = {H. S. M. Coxeter},
}

@article {JeWeYu17,
    AUTHOR = {Jerrard, Richard P. and Wetzel, John E. and Yuan, Liping},
     TITLE = {Platonic passages},
   JOURNAL = {Math. Mag.},
  FJOURNAL = {Mathematics Magazine},
    VOLUME = {90},
      YEAR = {2017},
    NUMBER = {2},
     PAGES = {87--98},
      ISSN = {0025-570X},
   MRCLASS = {52B10 (52A15)},
  MRNUMBER = {3626279},
MRREVIEWER = {Marek Lassak},
       DOI = {10.4169/math.mag.90.2.87},
       URL = {https://doi.org/10.4169/math.mag.90.2.87},
}

@article {ChYaZa18,
    AUTHOR = {Chai, Ying and Yuan, Liping and Zamfirescu, Tudor},
     TITLE = {Rupert property of {A}rchimedean solids},
   JOURNAL = {Amer. Math. Monthly},
  FJOURNAL = {American Mathematical Monthly},
    VOLUME = {125},
      YEAR = {2018},
    NUMBER = {6},
     PAGES = {497--504},
      ISSN = {0002-9890},
   MRCLASS = {52B10},
  MRNUMBER = {3806264},
MRREVIEWER = {Robert Davis},
       DOI = {10.1080/00029890.2018.1449505},
       URL = {https://doi.org/10.1080/00029890.2018.1449505},
}

@article {Hoffmann19,
    AUTHOR = {Hoffmann, Bal\'{a}zs},
     TITLE = {Rupert properties of polyhedra and the generalised {N}ieuwland
              constant},
   JOURNAL = {J. Geom. Graph.},
  FJOURNAL = {Journal for Geometry and Graphics},
    VOLUME = {23},
      YEAR = {2019},
    NUMBER = {1},
     PAGES = {29--35},
      ISSN = {1433-8157},
   MRCLASS = {52B10},
  MRNUMBER = {3982406},
}

@article {Lavau19,
    AUTHOR = {Lavau, G\'{e}rard},
     TITLE = {The truncated tetrahedron is {R}upert},
   JOURNAL = {Amer. Math. Monthly},
  FJOURNAL = {American Mathematical Monthly},
    VOLUME = {126},
      YEAR = {2019},
    NUMBER = {10},
     PAGES = {929--932},
      ISSN = {0002-9890},
   MRCLASS = {51M20 (52B10)},
  MRNUMBER = {4033228},
       DOI = {10.1080/00029890.2019.1656958},
       URL = {https://doi.org/10.1080/00029890.2019.1656958},
}

@thesis{Tonpho18,
doi = {http://cuir.car.chula.ac.th/handle/123456789/73138},
author = {Tonpho, Pongbunthit},
title = {Covering of objects related to {R}upert property},
year = {2018},
note = {Master Thesis, \url{http://cuir.car.chula.ac.th/handle/123456789/73138}}
}

@misc{Fredriksson22,
  doi = {10.48550/ARXIV.2210.00601},
  url = {https://arxiv.org/abs/2210.00601},
  howpublished = {\url{https://arxiv.org/abs/2210.00601}},
  author = {Fredriksson, Albin},
  title = {The triakis tetrahedron and the pentagonal icositetrahedron are {R}upert},
  publisher = {arXiv},
  year = {2022},
}

@misc{Scott2022,
      title={Two Sufficient Conditions for a Polyhedron to be (Locally) {R}upert}, 
      author={Evan Scott},
      year={2022},
      eprint={2208.12912},
      archivePrefix={arXiv},
      primaryClass={math.MG},
      note = {\href{https://arxiv.org/abs/2208.12912}{arxiv.org/abs/2208.12912}},
      url={https://arxiv.org/abs/2208.12912},   
}

\end{document}